\documentclass[hidelinks,11pt]{amsart}
\usepackage{amsmath,amsthm,amsfonts,amssymb}
\usepackage{parskip,fullpage}
\usepackage{longtable}
\usepackage{hyperref}
\usepackage[dvipsnames]{xcolor}
\usepackage{comment}
\usepackage{graphicx}
\usepackage{geometry}
\usepackage[shortlabels]{enumitem}
\usepackage{pgfplots}
\usepackage[table]{colortbl}
\pgfplotsset{compat=1.18}
\usepackage{subcaption}
\usepackage[american]{babel}
\usepackage{dsfont}
\hypersetup{
    colorlinks=true,
    citecolor=red
}
\usepackage[autostyle=true]{csquotes}
\usepackage[
backend=biber,
style=alphabetic,
sorting=nyt
]{biblatex}
\DeclareLabelalphaTemplate{
  \labelelement{
    \field[final]{shorthand}
    \field{label}
    \field[strwidth=3,strside=left,ifnames=1]{labelname}
    \field[strwidth=1,strside=left]{labelname}
  }
  \labelelement{
    \field[strwidth=4,strside=right]{year}
  }
}

\numberwithin{equation}{section}
\newtheorem{theorem}{Theorem}[section]
\newtheorem{lemma}[theorem]{Lemma}
\newtheorem{remark}[theorem]{Remark}

\newtheorem{corollary}[theorem]{Corollary}
\newtheorem{proposition}[theorem]{Proposition}

\newtheorem{example}[theorem]{Example}
\newtheorem{conjecture}[theorem]{Conjecture}

\newtheorem{problem}[theorem]{Problem}

\def\N{\mathbb{N}}

\def\R{\mathbb{R}}

\def\HH{\mathcal{H}}

\def\A{\mathcal{A}}

\newcommand{\suchthat}{\;\ifnum\currentgrouptype=16 \middle\fi|\;}

\makeatletter

\let\saved@setabstract\@setabstract

\renewcommand{\@setabstract}{%
  \begingroup
    \centering
    \@setaddresses
    \par
  \endgroup
  \global\let\addresses\@empty
  \saved@setabstract
}

\makeatother

\title{Constructions and reductions on the Hornich-Hlawka inequality and its generalizations}
\author{Seokbin Lee}
\address{School of Mathematics, Georgia Institute of Technology, 686 Cherry Street Atlanta, GA 30332, USA}
\email{slee3379@gatech.edu}

\subjclass[2020]{Primary 46B20; Secondary 46B04, 46B03, 52A40, 52A21}
\keywords{Hornich-Hlawka inequality, Banach space geometry, four-ball intersection, symmetric lens, seminorm, Nakano spaces, Sobolev spaces}

\begin{document}

\begin{abstract}
We investigate five open problems concerning the Hornich-Hlawka inequality and its generalizations posed by Marinescu and Niculescu. We give constructive answers to three of them. First, we construct an equivalent norm on $\ell^2$ with no three-dimensional Hornich-Hlawka subspace. Second, we present a Nakano space that is not Hornich-Hlawka. Third, we show that the $\alpha$-power generalization of the Hornich-Hlawka inequality fails to hold on nonzero normed spaces when $\alpha \in (1, 2) \cup (2, \infty)$. This implies that on every nonzero inner product space, the $\alpha$-power generalization of the Hornich-Hlawka inequality holds for all triples precisely when $\alpha \in (0,1] \cup \{2\}$. For the remaining two problems, we establish reductions. We prove that $\ell^p(N)$ is Hornich-Hlawka for all $N \geq 3$ and $2 < p \leq \log 3 /(\log 3 - \log 2)$ if and only if it is Hornich-Hlawka at $N = 3$ and $p = \log 3 /(\log 3 - \log 2)$. Finally, for $p(\cdot): (0, 1) \to [1, 2]$ measurable and integers $k \geq 0$, we establish a Sobolev-Nakano equivalence: we prove that $W^{k,p(\cdot)}(0,1)$, equipped with the sum of the Luxemburg norms of its derivatives of orders $0, \ldots, k$, is Hornich-Hlawka if and only if $L^{p(\cdot)}(0,1)$ is Hornich-Hlawka.
\end{abstract}
\maketitle

\section{Introduction} \label{sec: introduction}
For every $x, y, z \in \R^n$, the following inequality holds, called the \emph{Hornich-Hlawka inequality}:
\begin{align}
\lVert x \rVert_2 + \lVert y \rVert_2 + \lVert z \rVert_2 +\lVert x + y + z \rVert_2 \geq \lVert x + y \rVert_2 +  \lVert  y + z \rVert_2 +  \lVert z + x \rVert_2.
\end{align}
The inequality can be traced back to Hornich and Hlawka \cite{hornich1942inequality, hlawka1990inequality}, who stated and proved the inequality. Hlawka's elementary proof involves multiplying LHS - RHS by the LHS and invoking the \emph{Fr\'echet identity}
\begin{align} \label{eq: frechet}
\lVert x \rVert_2^2 + \lVert y \rVert_2^2 + \lVert z \rVert_2^2 +\lVert x + y + z \rVert_2^2 = \lVert x + y \rVert_2^2 +  \lVert  y + z \rVert_2^2 +  \lVert z + x \rVert_2^2,
\end{align}
which is equivalent to the parallelogram law. The Hornich-Hlawka inequality then reduces to the triangle inequality. This proof generalizes to all inner product spaces (and hence Hilbert spaces).

Let $X$ be a Banach space, and let $x, y, z \in X$. We say that $X$ is a \emph{Hornich-Hlawka space} if its norm satisfies the Hornich-Hlawka inequality
\begin{align}
\lVert x \rVert + \lVert y \rVert + \lVert z \rVert +\lVert x + y + z \rVert \geq \lVert x + y \rVert +  \lVert  y + z \rVert +  \lVert z + x \rVert \label{eq: hh inequality}
\end{align}
for all $x, y, z \in X$. In this case, we also say that $X$ has the \emph{Hornich-Hlawka property}, or simply that $X$ is \emph{Hornich-Hlawka}. A Hornich-Hlawka space is also called a \emph{quadrilateral space} in some literature \cite{smiley1964polygonal, simon1995geometric}, as the Hornich-Hlawka inequality is equivalent to an inequality on quadrilaterals.
Although every Hilbert space is Hornich-Hlawka, not every Banach space is Hornich-Hlawka.

Recently, Marinescu and Niculescu \cite{marinescu2025survey} conducted a broad survey of the Hornich-Hlawka inequality and its generalizations; as illustrated there, the Hornich-Hlawka inequality enjoys connections to several areas of mathematics, including operator theory \cites{lindenstrauss1968operators, audenaert2017problems}, probability and statistics \cites{levy1939addition, ferguson1962cauchy}, and convex geometry \cites{levi1949linear, vasic1969inequalities, mitrinovi1993classical}. The authors of \cite{marinescu2025survey} also posed five open problems in Banach space theory relating to the Hornich-Hlawka inequality and its generalizations that were described as seeming ``of greatest interest." We follow their conventions and number these problems as Problems 1, 2, 3, 4, and 5.  

In this paper, we explore each of these problems in detail. 

\subsection{Main results} 
First, for Problem 1, we prove two reductions (Theorem \ref{thm: all lpn} and Corollary \ref{cor: p0}) that imply the following:

\begin{theorem} \label{thm: endpoint}
Let $p_0 =\log 3 /(\log 3 - \log 2) \approx 2.7095\cdots$. Then the space $\ell^p(N)$ is Hornich-Hlawka for all $N \geq 3$ and $2 < p \leq p_0$ if and only if $\ell^{p_0}(3)$ is Hornich-Hlawka. 
\end{theorem}

Thus, if $\ell^{p_0}(3)$ were proven to be Hornich-Hlawka, then we would obtain a complete classification of finite-dimensional $\ell^p$ spaces that are Hornich-Hlawka (see Table \ref{table: table}). We also provide two convex-geometric reformulations of the Hornich-Hlawka inequality in Theorem \ref{thm: four ball equivalence}.

For Problem 2, we prove the existence of a renorming on $\ell^2$ that admits an infinite-dimensional Banach space without a three-dimensional Hornich-Hlawka subspace:

\begin{theorem} \label{thm: l^2 equivalent}
There exists an equivalent norm $\lVert \cdot \rVert_X$ on the real Hilbert space $\ell^2$ such that no three-dimensional subspace of $X = (\ell^2, \lVert \cdot \rVert_X)$ is Hornich-Hlawka. 
\end{theorem}

For Problem 3, we exhibit an explicit function $p(\cdot):(0, 1) \to [1, 2]$ such that the Nakano space $L^{p(\cdot)}(0, 1)$ is not Hornich-Hlawka:

\begin{theorem} \label{thm: nakano counterexample}
Define the function
\[
p(t) = \begin{cases}
1, \quad  &0 < t < \frac35 \\
5t - 2, \quad &\frac35 \leq t \leq \frac45 \\
2, \quad &\frac45 < t < 1
\end{cases}.
\]
Then $L^{p(\cdot)}(0, 1)$ is not Hornich-Hlawka.
\end{theorem}

For Problem 4, we provide an explicit reduction of a Sobolev space $W^{k, \, p(\cdot)}(0, 1)$ having the Hornich-Hlawka property to the associated Nakano space $L^{p(\cdot)}(0, 1)$ when assuming a particular choice of norm for the Sobolev space.

\begin{theorem} \label{thm: sobolev nakano}
Let $p(\cdot): (0, 1) \to [1, 2]$ be measurable and let $k \in \mathbb{N} \cup \{0\}$. With the choice of the norm \eqref{eq: sobolev norm} for $W^{k, \, p(\cdot)}(0, 1)$, the Sobolev space $W^{k, \, p(\cdot)}(0, 1)$ is Hornich-Hlawka if and only if $L^{p(\cdot)}(0, 1)$ is Hornich-Hlawka.
\end{theorem}

For Problem 5, we give two elementary examples that show that the $\alpha$-power generalization of the Hornich-Hlawka inequality cannot hold for any nonzero normed space when $\alpha \in (1, 2) \cup (2, \infty)$. As the statements and proofs of both examples are short and elementary, we refer the reader to Theorems \ref{thm: x=y=z} and \ref{thm: x=y=-z} in Section \ref{sec: problem 5} for the details.

The remainder of the paper is organized as follows. Section \ref{sec: notation} provides the common notation that will be used throughout the paper. Section \ref{sec: preliminaries} contains a review of preliminary results. Sections \ref{sec: problem 1} through \ref{sec: problem 5} discuss our original contributions to each numbered problem in order. 

\section{Notation and Terminology} \label{sec: notation}

When writing the variables $p$ and $N$, we always assume that $1 \leq p \leq \infty$ and $N \in \N$ unless their ranges are otherwise specified. 

Following the notation of \cite{marinescu2025survey}, we denote by $\ell^p(N) = (\R^N, \, \lVert \cdot \rVert_p)$ the space of $N$-tuples equipped with the $p$-norm, and we denote by $\ell^p$ the infinite-dimensional sequence-space analogue. The ambient field is always assumed to be the field of real numbers $\R$.

The triple $(\Omega, \, \A, \, \mu)$ will always denote a measure space. In the context of probability spaces, we will write $\mathbb{P}$ instead of $\mu$ to emphasize that the measure function is a probability measure.

We write $L^p(\Omega, \, \mu)$ to denote the space of $p$-th power Lebesgue integrable functions $f: \Omega \to \R$ with respect to $\mu$, with the understanding that $p = \infty$ means we take the essential supremum norm. Moreover, when $X$ is a Banach space, we write $L_X^p(\Omega, \,\mu)$ to denote the space of $p$-th power Bochner integrable functions $f: \Omega \to X$ with respect to $\mu$, again separating the $p = \infty$ case to mean the essential supremum. In this setting, we call $X$ the \emph{codomain space} of $L_X^p(\Omega, \,\mu)$. In particular, if $X = \R$, then we recover the set of $p$-th power Lebesgue integrable functions $L_\R^p(\Omega, \, \mu) = L^p(\Omega, \, \mu)$.

We write $L^{p(\cdot)}(0, 1)$ to denote a Nakano space, where $p(\cdot): (0, 1) \to[ 1, 2]$ is a variable exponent function that is measurable. Similarly, we write $W^{k, p(\cdot)}(0, 1)$ to denote a Sobolev space, where $k$ is a nonnegative integer and $p(\cdot): (0, 1) \to[ 1, 2]$ is a variable exponent function that is measurable. In particular, if $k = 0$, then these two spaces coincide.

Throughout this paper, we assume the standard Luxemburg norm
\[
\lVert f \rVert_{p(\cdot)} = \inf \left\{\lambda > 0 : \int_0^1 \left|\frac{f(t)}{\lambda}\right|^{p(t)} \, dt \leq 1\right\}
\]
for the Nakano spaces $L^{p(\cdot)}(0, 1)$, and the norm
\begin{align} \label{eq: sobolev norm}
\lVert f \rVert_{k, \, p(\cdot)} = \sum_{j=0}^k \left\lVert f^{(j)} \right\rVert_{p(\cdot)} 
\end{align}
for the Sobolev spaces $W^{k, \, p(\cdot)}(0, 1)$, where we interpret the derivatives as weak derivatives. We remark that although there are several topologically equivalent norms for Sobolev spaces, our specific choice will allow us to give an elegant partial answer to Problem 4 as in Theorem \ref{thm: sobolev nakano}.

Finally, let $p_0$ represent the constant
\[
p_0 = \frac{\log 3}{\log 3 - \log 2} \approx 2.7095\cdots.
\]

\section{Preliminaries} \label{sec: preliminaries}

\subsection{\texorpdfstring{The Hornich-Hlawka property of $\ell^p(N)$ spaces}{The Hornich-Hlawka property of lp(N) spaces}} \label{subsec: lp(N) summary}

First, we establish the current standings of which $\ell^p(N)$ spaces are known to be Hornich-Hlawka, known not to be Hornich-Hlawka, or unknown. Table \ref{table: table} summarizes these standings.

\begin{center}
\begin{table}[htbp] 
\caption{Landscape of $\ell^p(N)$ spaces with regards to the Hornich-Hlawka property. The top row (Corollary \ref{cor: n >= 2}) and leftmost column (Corollary \ref{cor: p <= 2}) both follow from $L^1$-embeddability (Corollary \ref{cor: l1 space embedding HH}); the bottom right cell is due to Corollary \ref{cor: bad regime}.}\label{table: table}
\begin{tabular}{|c|c|c|c|}
\hline
 & $1 \leq p \leq 2$ & $2 < p \leq (\log3) /(\log3 - \log2)$ &  $(\log3) /(\log3 - \log2) < p \leq \infty$\\
\hline
$N \leq 2$ & \cellcolor{green!30}Hornich-Hlawka & \cellcolor{green!30}Hornich-Hlawka & \cellcolor{green!30}Hornich-Hlawka \\
\hline
$N \geq 3$ & \cellcolor{green!30}Hornich-Hlawka & \cellcolor{gray!30}Unknown & \cellcolor{red!30}not Hornich-Hlawka \\
\hline
\end{tabular}
\end{table}
\end{center}

We will survey the known cases (the green and red cells of Table \ref{table: table}). The organization and exposition will closely mirror those of \cite{marinescu2025survey}.

The key idea behind the $\ell^p(N)$ spaces that are known to be Hornich-Hlawka is that they can be embedded into an $L^1$ space, which is itself Hornich-Hlawka. This is formally described in Proposition \ref{prop: bochner hh} and Corollary \ref{cor: l1 space embedding HH}.

\begin{proposition} \cite[Lemma~1]{marinescu2025survey} \label{prop: bochner hh} 
If $X$ is a Hornich-Hlawka space, then every space $L^1_X(\Omega, \, \mu)$ is Hornich-Hlawka. In particular, every space $L^1(\Omega, \, \mu)$ is Hornich-Hlawka.
\end{proposition}

\begin{corollary}\label{cor: l1 space embedding HH}
Every space that embeds linearly and isometrically into an $L^1$ space is Hornich-Hlawka.
\end{corollary}

\begin{remark} \label{rem: false converse}
The converse to Corollary \ref{cor: l1 space embedding HH} is false; Neyman \cite[Remark~3.5]{neyman1984representation} showed that there is no finite list of inequalities that characterizes embeddability in $L^1(0, 1)$. 
So we cannot conclude that a space that is not linearly and isometrically embeddable into an $L^1$ space is not Hornich-Hlawka.
\end{remark}

Corollary \ref{cor: l1 space embedding HH} together with the following proposition allows us to verify the Hornich-Hlawka property of some $\ell^p(N)$ spaces.

\begin{proposition} \cites[Theorem~1]{ferguson1962cauchy}[Theorem~2]{herz1963negativedefinite}[Theorem~3]{lindenstrauss1964extension} If $X$ is a two-dimensional Banach space, then $X$ is a Hornich-Hlawka space.
\end{proposition}

\begin{corollary} \label{cor: n >= 2} If $N \leq 2$, then $\ell^p(N)$ is a Hornich-Hlawka space.
\end{corollary}

\begin{proposition} \label{prop: embedding} \cites[Theorems~1, 2]{bretagnolle1965lp}[Theorem~1.1]{mcnicholl1965lp} If $1 \leq r \leq p \leq 2$ and $N \in \N$, then the spaces $\ell^p(N)$, $\ell^p$, and $L^p(0, 1)$ embed linearly and isometrically into $L^r(0, 1)$.
\end{proposition}

The following corollary follows from Propositions \ref{prop: bochner hh} and \ref{prop: embedding}.

\begin{corollary} \label{cor: p <= 2} If $1 \leq p \leq 2$, then $\ell^p(N)$, $\ell^p$, and $L^p(0, 1)$ are all Hornich-Hlawka spaces.
\end{corollary}

We also have some elementary negative results. We may verify the following by a straightforward calculation.

\begin{proposition} \cite[Proposition~2]{witsenhausen1973zonoid} \label{prop: bad spaces}
If $p_0 < p \leq \infty$, then $\ell^p(3)$ is not a Hornich-Hlawka space. In particular,
\[
x = (-1, \, 1, \, 1), \quad y = (1, \, -1, \, 1), \quad z = (1, \, 1, \, -1)
\]
provides a counterexample to the Hornich-Hlawka inequality \eqref{eq: hh inequality}.
\end{proposition}

The following proposition allows us to extend the results of Proposition \ref{prop: bad spaces} to $\ell^p(N)$ spaces where $N \geq 3$.

\begin{proposition} \label{prop: zero padding}
Let $N \geq 3$. If $\ell^p(3)$ is not Hornich-Hlawka, then $\ell^p(N)$ is not Hornich-Hlawka.
\end{proposition}

\begin{proof}
Assume $x, y, z \in \ell^p(3)$ violate the Hornich-Hlawka inequality \eqref{eq: hh inequality}. Then the zero-padded tuple
\[
\hat{x} = (x, \underbrace{0, \ldots, 0}_{N-3}) \in \ell^p(N)
\]
satisfies $\lVert \hat{x} \rVert_{\ell^p(N)} = \lVert x \rVert_{\ell^p(3)}$. By padding $y$ and $z$ with zeroes analogously, we see that the tuples
\[
(x, \underbrace{0, \ldots, 0}_{N-3}), \, (y, \underbrace{0, \ldots, 0}_{N-3}), \, (z, \underbrace{0, \ldots, 0}_{N-3}) \in \ell^p(N)
\]
also violate the Hornich-Hlawka inequality \eqref{eq: hh inequality}.
\end{proof}

As a consequence of Propositions \ref{prop: bad spaces} and \ref{prop: zero padding}, we obtain the following corollary.

\begin{corollary} \label{cor: bad regime}
If $p_0 < p \leq \infty$ and $N \geq 3$, then $\ell^p(N)$ is not a Hornich-Hlawka space. 
\end{corollary}

\subsection{An identity in fractional calculus} The proof of Theorem \ref{thm: p=p0 reduction} requires the following lemma.

\begin{lemma} \label{lemma: balakrishnan}
Let $0 < \beta < 1$ and $x \geq 0$. The following identity holds:
\begin{align}
x^\beta = \frac{\beta}{\Gamma(1 - \beta)}\int_0^\infty (1 - e^{-tx})t^{-\beta-1}\,dt. \label{eq: balakrishnan}
\end{align}
\end{lemma}

The identity appears (in various forms) in \cites{weyl1917differential, marchaud1927differences, schoenberg1938metric}. Balakrishnan's formula \cite{balakrishnan1960fractional} replaces the scalar $x$ in \eqref{eq: balakrishnan} with a specific closed linear operator; modern expositions such as \cite{handbookfractional} motivate Balakrishnan's formula by first explicitly stating its scalar form \eqref{eq: balakrishnan}.

Lemma \ref{lemma: balakrishnan} admits an elementary, self-contained proof, which we provide here.

\begin{proof}
If $x = 0$, then both sides of \eqref{eq: balakrishnan} are zero, so we assume $x > 0$. Let
\[
I = \int_0^\infty (1 - e^{-tx})t^{-\beta-1}\,dt.
\]
Integration by parts yields
\begin{align} \label{eq: int by parts}
I = \left[(1 - e^{-tx})\left(-\frac{t^{-\beta}}\beta\right)\right]_{t=0}^{t=\infty} - \int_0^\infty xe^{-tx}\left(-\frac{t^{-\beta}}\beta\right)\,dt.
\end{align}
Evaluating the boundary term, as $t \to \infty$ we have
\[
(1 - e^{-tx})\left(-\frac{t^{-\beta}}\beta\right) \longrightarrow 0
\]
and as $t \to 0$ we have
\[
(1 - e^{-tx})\left(-\frac{t^{-\beta}}\beta\right) = \left(tx + o(t)\right)\left(-\frac{t^{-\beta}}\beta\right) = -\frac{x}\beta t^{1-\beta} + o(t^{1-\beta}) \longrightarrow 0.
\]
Hence, \eqref{eq: int by parts} simplifies to
\[
I = \frac{x}{\beta}\int_0^\infty t^{-\beta}e^{-tx}\,dt
\]
and the substitution $s = tx$ yields
\begin{align*}
I = \frac{x}{\beta}\int_0^\infty \left(\frac{s}{x}\right)^{-\beta}e^{-s}\frac1x\,ds = \frac{x^\beta}{\beta} \int_0^\infty s^{-\beta}e^{-s}\, ds = \frac{x^\beta}{\beta}\Gamma(1 - \beta).
\end{align*}
Substituting our definition of $I$, this rearranges to the promised identity.
\end{proof}

\subsection{A duality lemma}
The following lemma is standard in functional analysis and simplifies our proof of Theorem \ref{thm: four ball equivalence}.

\begin{lemma} \label{lemma: bounded below operators} Let $X, Y$ be Banach spaces, let $T: X \to Y$ be a bounded linear operator, and let $C$ be a positive constant. The following are equivalent:

\begin{enumerate}
    \item For all $x \in X$, we have
    \[\lVert Tx \rVert_Y \geq C \lVert x \rVert_X;\]
    \item For all $x^* \in X^*$, there is $y^* \in Y^*$ such that $T^*y^* = x^*$ and 
    \[
    \lVert y^* \rVert_{Y^*} \leq \frac1C \lVert x^* \rVert_{X^*}.
    \]
\end{enumerate}
\end{lemma}

We omit its proof and refer to \cite[Theorem~4.14]{rudin1991functional} and its proof for a detailed discussion.

\section{\texorpdfstring{Problem 1: The Hornich-Hlawka property of $\ell^p(N)$}{Problem 1: The Hornich-Hlawka property of lp(3)}} \label{sec: problem 1}

We first restate Problem 1 as a conjecture according to \cite{marinescu2025survey}.

\begin{conjecture} \label{conj: lp0(3) not hh}
The space $\ell^p(3)$ is not Hornich-Hlawka for all $2 < p \leq p_0$.
\end{conjecture}

Combined with Proposition \ref{prop: zero padding}, Conjecture \ref{conj: lp0(3) not hh} would settle the shaded area in Table \ref{table: table} negatively.

\subsection{\texorpdfstring{Reduction to three-dimensional $\ell^p$ spaces}{Reduction to three-dimensional lp spaces}}
In this subsection, our aim is to show Theorem \ref{thm: all lpn}, which states the converse of Proposition \ref{prop: zero padding}. Therefore, the analysis of the shaded area in Table \ref{table: table} can always be reduced to merely inspecting what happens when $N = 3$ for every fixed $p$ (Corollary \ref{cor: n >= 3 hh}).

We first state and prove an equivalent inequality that arises from \eqref{eq: hh inequality} after a change of variables.

\begin{proposition}
Let $X$ be a Banach space with norm $\lVert \cdot \rVert$. Then $X$ is a Hornich-Hlawka space (i.e. \eqref{eq: hh inequality} is satisfied for all $x, y, z \in X$) if and only if
\begin{align}
\lVert u + v + w \rVert + \lVert u + v - w \rVert + \lVert u - v + w \rVert  + \lVert - u + v + w \rVert \geq 2(\lVert u \rVert + \lVert v \rVert + \lVert w \rVert) \label{eq: uvw rewrite}
\end{align}
is satisfied for all $u, v, w \in X$.
\end{proposition}

\begin{proof}
We may substitute
\[
u = x + y, \quad v = y + z, \quad w = z + x,
\]
in $\eqref{eq: uvw rewrite}$ to get $\eqref{eq: hh inequality}$  after invoking homogeneity of the norm and dividing both sides by 2. Conversely, we may substitute
\[
x = \frac{u - v + w}2, \quad y = \frac{u + v - w}2, \quad z = \frac{-u + v + w}2.
\]
in $\eqref{eq: hh inequality}$ and proceed similarly to get \eqref{eq: uvw rewrite}. 
\end{proof}

In order to write \eqref{eq: uvw rewrite} more concisely, we will define the set of sign patterns
\[
\HH =\{(1, \, 1, \, 1), \,(1, \, 1, \, -1), \, (1, \, -1, \, 1), \, (-1, \, 1, \,1)\},
\]
so that \eqref{eq: uvw rewrite} can be rewritten as
\[
\sum_{\sigma \in \HH} \lVert \sigma_1 u + \sigma_2 v + \sigma_3 w \rVert \geq 2(\lVert u \rVert + \lVert v \rVert + \lVert w \rVert).
\]

In addition to the usual Hornich-Hlawka inequality \eqref{eq: hh inequality}, we will also call the equivalent inequality \eqref{eq: uvw rewrite} the Hornich-Hlawka inequality; for the remainder of this paper, we will explicitly write \eqref{eq: hh inequality} or \eqref{eq: uvw rewrite} for disambiguation when relevant.

When studying the Hornich-Hlawka property of $\ell^p(N)$ spaces, it is natural to ponder its behavior as $N \to \infty$ when $p$ is fixed. The following theorem shows that the behavior is in fact independent of $N$ so long as $N \geq 3$.

\begin{theorem} \label{thm: all lpn}
Let $1 \leq p \leq \infty$. Then $\ell^p(3)$ is Hornich-Hlawka if and only if $\ell^p(N)$ is Hornich-Hlawka for every $N \geq 3$.
\end{theorem}

Our strategy is to assume that $\ell^p(N)$ is not Hornich-Hlawka for some $N$, so that the inequality in \eqref{eq: uvw rewrite} is flipped. By an appropriate parametrization, we will treat the left-hand side as a concave function over a compact polytope where the right-hand side is fixed, and use the Bauer Maximum Principle to show that the left-hand side is minimized at a vertex of that polytope. The coordinates of that vertex will then be cast into $\ell^p(3)$, proving that $\ell^p(3)$ is also not Hornich-Hlawka.

Our proof, though originally formulated with AI assistance and simplified and rewritten solely by the author, was independently generated by \cite{palomar-2026-09-25-000006-v1} using essentially the same ideas and was formalized in Lean. Their work investigates a generalization of the Hornich-Hlawka inequality for the Schatten $p$-norm of matrices \cite[Problem~6]{audenaert2017problems}.

The focus of \cite{palomar-2026-09-25-000006-v1} is on the Lean formalization of a partial resolution of this problem, which includes a proof of Theorem \ref{thm: all lpn} as an intermediate result. Though formalized in Lean, there is not a human-readable manuscript detailing the proof as of the date this paper is written; as such, we provide the proof below.

\begin{proof}[Proof of Theorem \ref{thm: all lpn}] The ``if" direction is clear, so we focus on the ``only if" direction. Also, if $p = \infty$, then by Corollary \ref{cor: bad regime} the theorem is vacuously true, so we assume that $1 \leq p < \infty$.

Assume that $\ell^p(N)$ is not a Hornich-Hlawka space for some $N \geq 3$. 
Writing the negation of the Hornich-Hlawka inequality \eqref{eq: uvw rewrite}, we get that there are $u, \, v, \, w \in \ell^p(N)$ such that
\begin{align}
\lVert u + v + w \rVert_p + \lVert u + v - w \rVert_p + \lVert u - v + w \rVert_p  + \lVert - u + v + w \rVert_p < 2(\lVert u \rVert_p + \lVert v \rVert_p + \lVert w \rVert_p). \label{eq: uvw failure}
\end{align}
For each $j \in \{1, \, \ldots, \, N\}$, let $u_j, \, v_j,\, w_j$ be the $j$-th coordinates of $u, \, v, \, w$ respectively, and let 
\[
t_j = (u_j, \, v_j, \, w_j) \in \R^3, \quad \lambda_j = |u_j|^p + |v_j|^p + |w_j|^p.
\]
By discarding each $t_j$ that is the zero triplet, we assume without loss that each $t_j$ is not the zero triplet and write
\[
r_j = \lambda_j^{-1/p}t_j \in \R^3,
\]
so  that $r_j$ lies on the $\ell^p(3)$-unit sphere \[S_p = \{(a, b, c) \in \R^3 : |a|^p + |b|^p + |c|^p = 1\},\]
because
\begin{align}
|r_{j1}|^p + |r_{j2}|^p + |r_{j3}|^p = \lVert r_j \rVert_p^p = \lambda_j^{-1} \lVert t_j \rVert_p^p = 1. \label{eq: rj unit sphere}
\end{align}
Let 
\[
g_j = \left(|r_{j1}|^p, \, |r_{j2}|^p, \, |r_{j3}|^p\right) \in \R^3,
\]
and define the moment vector
\[
M = (M_1, \, M_2, \, M_3), \quad M_k = \sum_{j=1}^N \lambda_j g_{jk}, \quad k \in \{1, \, 2, \, 3\}.
\]
By construction, we have
\[
M_1 = \lVert u \rVert_p^p, \quad M_2 = \lVert v \rVert_p^p, \quad M_3 = \lVert w \rVert_p^p.
\]
Now consider the polytope
\[
P = \left\{\alpha = (\alpha_1, \, \ldots, \, \alpha_N) : \alpha_j \geq 0, \quad \sum_{j=1}^N \alpha_j g_j = M\right\}.
\]
We claim that $P$ is nonempty and compact. It is nonempty because $\lambda = (\lambda_1, \, \ldots, \, \lambda_N) \in P$. Also, by \eqref{eq: rj unit sphere} we have that $
g_{j1} + g_{j2} + g_{j3} = 1 $, so if $\alpha \in P$, then
\[
\sum_{j=1}^N \alpha_j = \sum_{j=1}^N \alpha_j (g_{j1} + g_{j2} + g_{j3}) = M_1 + M_2 + M_3.
\]
Thus, $P$ is bounded. It is also closed, so it is compact.

Define over $P$ the function
\begin{align}
\Phi(\alpha) = \sum_{\sigma \in \HH}\left(\sum_{j=1}^N \alpha_j \left|\sigma_1 r_{j1} + \sigma_2 r_{j2} + \sigma_3 r_{j3}\right|^p\right)^{1/p}, \label{eq: phi definition}    
\end{align}
so that \eqref{eq: uvw failure} can be rewritten as
\begin{align}
\Phi(\lambda) < 2\left(M_1^{1/p} + M_2^{1/p} + M_3^{1/p}\right).\label{eq: phi lambda}
\end{align}

Inspecting \eqref{eq: phi definition}, each expression inside a $p$-th root is linear in $\alpha$, while the map $s \mapsto s^{1/p}$ is concave for $p \geq 1$. Therefore, $\Phi$ is a continuous concave function in $\alpha$. A continuous concave function over a compact polytope attains its minimum at a vertex, so we may let $\alpha^* \in P$ be a vertex of $P$ that minimizes $\Phi$. 

We claim that $\alpha^*$ has at most 3 positive coordinates. Suppose to the contrary that there are at least 4 positive coordinates of $\alpha^*$; the corresponding vectors $g_j$ in $\R^3$ would be linearly dependent. Consequently, there is a nonzero vector $h = (h_1, \ldots, h_N)$ supported only on those positive coordinates such that
\[
\sum_{j} h_j g_j  = 0.
\]
Now take a sufficiently small $\varepsilon > 0$ such that $\alpha^* + \varepsilon h$ and $\alpha^* - \varepsilon h$ are both nonnegative vectors. Both vectors remain in $P$, because
\[
\sum_{j}(\alpha^*_j \pm \varepsilon h_j)g_j = M.
\]
Thus, we have that
\[
\alpha^* = \frac12(\alpha^* + \varepsilon h) + \frac12(\alpha^* - \varepsilon h)
\]
is a nontrivial convex decomposition in $P$ because $h \neq 0$. This contradicts the assumption that $\alpha^*$ is a vertex of $P$.

Hence, we may let the positive coordinates of $\alpha^*$ be indexed by $j_1, \, \ldots, \, j_s$, where $s \leq 3$, and we write
\[
\alpha^* = \sum_{k=1}^s \beta_k e_{j_k}, \quad s \leq 3
\]
with $\beta_k \geq 0 $ such that the moment identity is retained:
\[
M=\sum_{k=1}^{s}\beta_kg_{j_k};
\]
we have
\begin{align}
\Phi(\alpha^*) \leq \Phi(\lambda). \label{eq: alpha star}
\end{align}
Write $
r_{j_k} = (a_k, b_k, c_k) \in S_p$ and define $u', v', w' \in \ell^p(3)$ by
\[
u'_k = \begin{cases}
\beta_k^{1/p} a_k ,&k \leq s\\
0, &k > s
\end{cases}, \quad v'_k = \begin{cases}
\beta_k^{1/p} b_k ,&k \leq s\\
0, &k > s
\end{cases}, \quad w'_k = \begin{cases}
\beta_k^{1/p} c_k ,&k \leq s\\
0, &k > s
\end{cases}, \quad 1 \leq k \leq 3.
\]
We have
\[
\lVert u' \rVert_p^p = M_1, \quad \lVert v' \rVert_p^p = M_2, \quad \lVert w' \rVert_p^p = M_3,
\]
and for every sign pattern $\sigma \in \HH$, we have
\begin{align}
\lVert \sigma_1 u' + \sigma_2 v' + \sigma_3 w' \rVert_p^p = \sum_{k=1}^s \beta_k |\sigma_1 a_k + \sigma_2 b_k + \sigma_3 c_k|^p. \label{eq: all sigma patterns}
\end{align}
Combining  \eqref{eq: phi definition}, \eqref{eq: phi lambda},  \eqref{eq: alpha star}, and \eqref{eq: all sigma patterns} gives us
\[
\sum_{\sigma \in \HH} \lVert\sigma_1 u' + \sigma_2 v' + \sigma_3 w' \rVert_p = \Phi(\alpha^*) \leq \Phi(\lambda) < 2\left(M_1^{1/p} + M_2^{1/p} + M_3^{1/p}\right)
\]
which expands out to
\[
\lVert u' + v' + w' \rVert_p + \lVert u' + v' - w' \rVert_p + \lVert u' - v' + w' \rVert_p + \lVert -u' + v' + w' \rVert_p  < 2(\lVert u' \rVert_p + \lVert v' \rVert_p + \lVert w' \rVert_p).
\]
We see that $u', v', w' \in \ell^p(3)$ violate the Hornich-Hlawka inequality \eqref{eq: uvw rewrite}, so $\ell^p(3)$ is not a Hornich-Hlawka space.
\end{proof}

\begin{corollary} \label{cor: n >= 3 hh}
Let $N \geq 3$. Then $\ell^p(N)$ is Hornich-Hlawka if and only if $\ell^p(3)$ is Hornich-Hlawka.
\end{corollary}

\begin{proof}
If $\ell^p(3)$ is Hornich-Hlawka, then Theorem \ref{thm: all lpn} implies that $\ell^p(N)$ is Hornich-Hlawka. Conversely, if $\ell^p(3)$ is not Hornich-Hlawka, then Proposition \ref{prop: zero padding} implies $\ell^p(N)$ is not Hornich-Hlawka. 
\end{proof}

Theorem \ref{thm: all lpn} may now be invoked to prove a generalization to $L^p$ spaces.

\begin{theorem} \label{thm: n=3 reduction}
Let $1 \leq p \leq \infty$. Then $\ell^p(3)$ is Hornich-Hlawka if and only if  $L^p(\Omega, \, \mu)$ is Hornich-Hlawka for every measure space $(\Omega, \, \A, \, \mu)$.
\end{theorem}

\begin{proof} Again, the ``if" direction is clear, and the case $p = \infty$ is vacuously true for the ``only if" direction, so we look at the ``only if" direction for $1 \leq p < \infty$.

Assume that there is a space $L^p(\Omega, \, \mu)$ that fails to be Hornich-Hlawka. Let $u, v, w \in L^p(\Omega, \, \mu)$ satisfy 
\begin{align}
\lVert u + v + w \rVert + \lVert u + v - w \rVert + \lVert u - v + w \rVert  + \lVert - u + v + w \rVert < 2\left(\lVert u \rVert + \lVert v \rVert + \lVert w \rVert\right). \label{eq: lp violation}
\end{align}

Simple functions are dense in $L^p(\Omega, \, \mu)$. Since all seven norm expressions in \eqref{eq: lp violation} depend continuously on $u, \, v,  \, w$, we may replace $u, \, v, \, w$ by sufficiently close simple functions that maintain the inequality. By taking a common refinement of their finite measurable partitions, we may write
\[
u = \sum_{j=1}^N u_j \mathbf{1}_{E_j}, \quad v = \sum_{j=1}^N v_j \mathbf{1}_{E_j}, \quad w = \sum_{j=1}^N w_j \mathbf{1}_{E_j},
\]
where $E_1, \ldots, E_N$ are pairwise disjoint. By discarding cells $E_j$ on which all three values of $u, \, v, \, w$ vanish, we may assume without loss that at least one of $u, \, v, \, w$ is nonzero on each cell $E_j$. Since each function $u, \, v, \, w$ is in $L^p(\Omega, \, \mu)$, the measure of each $E_j$ is finite. 

Define vectors $\hat{u}, \, \hat{v}, \, \hat{w} \in \ell^p(N)$ by
\[
\hat{u}_j = \mu(E_j)^{1/p}u_j, \quad \hat{v}_j = \mu(E_j)^{1/p}v_j,  \quad\hat{w}_j = \mu(E_j)^{1/p}w_j,
\]
padding each vector with zeroes if necessary if there are fewer than three cells $E_j$.
For all scalars $\alpha, \, \beta, \, \gamma \in \R$, we have
\begin{align*}
\lVert \alpha u + \beta v + \gamma w \rVert_{L^p(\Omega, \,\mu)}^p &= \sum_{j=1}^N \mu(E_j) |\alpha u_j + \beta v_j + \gamma w_j|^p \\
&= \sum_{j=1}^N  |\alpha \hat{u}_j + \beta \hat{v}_j + \gamma \hat{w}_j|^p = \lVert \alpha \hat{u} + \beta \hat{v} + \gamma \hat{w} \rVert_{\ell^p(N)}^p.
\end{align*}
Therefore, every norm appearing in \eqref{eq: lp violation} is preserved exactly. This means that $\hat{u}, \, \hat{v}, \, \hat{w} \in \ell^p(N)$ constitute a counterexample to the Hornich-Hlawka inequality \eqref{eq: uvw rewrite}. By Theorem \ref{thm: all lpn}, the space $\ell^p(3)$ is not Hornich-Hlawka.
\end{proof}

\subsection{\texorpdfstring{Reduction to the space $\ell^{p_0}(3)$}{Reduction to the space lp0(3)}}

We also prove the following partial reduction for three-dimensional $\ell^p$ spaces: a positive resolution for the Hornich-Hlawka property of $\ell^{p_0}(3)$ would imply that of all $\ell^p(3)$ in the range $2 < p \leq p_0$.

\begin{theorem} \label{thm: p=p0 reduction}
Let $1 < r < p < \infty$. If $\ell^p(3)$ is Hornich-Hlawka, then $\ell^r(3)$ is Hornich-Hlawka.
\end{theorem}

\begin{corollary} \label{cor: p0}
If $\ell^{p_0}(3)$ is Hornich-Hlawka, then $\ell^p(3)$ is Hornich-Hlawka for all $2 < p \leq p_0$.
\end{corollary}

The proof of Theorem \ref{thm: p=p0 reduction} is motivated by the original proof of Proposition \ref{prop: embedding} due to \cite{bretagnolle1965lp} and also relies on properties of stable random variables. We state and prove these properties before proceeding with the proof of Theorem \ref{thm: p=p0 reduction}.

From now on, let $0 < \alpha < 1$ and define $\nu_\alpha$ to be the positive $\alpha$-stable probability measure on $[0, \infty)$ characterized by
\[
\int_0^\infty e^{-ts} \, d\nu_\alpha(s) = e^{-t^\alpha}, \quad t \geq 0.
\]
The existence of this distribution follows, for instance, from \cite[Example~3.16]{janson2022stable} with $\lambda = 1$, and uniqueness follows from uniqueness of Laplace transforms. Also, let $S$ be a random variable with this distribution, i.e. \[
\mathbb{E}[e^{-tS}] = e^{-t^\alpha}.
\]

We need an auxiliary result about $\alpha$-stable random variables.

\begin{lemma} \label{lemma: finite mean}
If $0 < \beta < \alpha$, then $\mathbb{E}[S^{\beta}] < \infty$.
\end{lemma}

\begin{proof}
First, we put $x = S$ in \eqref{eq: balakrishnan} to get
\[
S^\beta = \frac{\beta}{\Gamma(1 - \beta)}\int_0^\infty (1 - e^{-tS})t^{-\beta-1}\,dt.
\]
Now we take expectations of both sides. Since the integrand on the right-hand side is nonnegative, Tonelli's Theorem allows us to interchange the expectation and integral to arrive at
\begin{align*}
\mathbb{E}[S^\beta] &= \frac{\beta}{\Gamma(1 - \beta)}\int_0^\infty \left(1 -\mathbb{E}[e^{-tS}]\right)t^{-\beta-1}\,dt \\
&= \frac{\beta}{\Gamma(1 - \beta)}\int_0^\infty \left(1 -e^{-t^\alpha}\right)t^{-\beta-1}\,dt \\
&= \frac{\beta}{\Gamma(1 - \beta)}\int_0^1\left(1 -e^{-t^\alpha}\right)t^{-\beta-1}\,dt + \frac{\beta}{\Gamma(1 - \beta)}\int_1^\infty \left(1 -e^{-t^\alpha}\right)t^{-\beta-1}\,dt
\end{align*}
Now it suffices to argue that both integrals on the right-hand side are finite. For $0 \leq t \leq 1$, we have 
\[
1 - e^{-t^\alpha} \leq t^{\alpha},
\]
so
\[
\int_0^1 \left(1 -e^{-t^\alpha}\right)t^{-\beta-1}\,dt \leq \int_0^1 t^{\alpha - \beta - 1}\, dt < \infty.
\]
because $\beta < \alpha$. For $t \geq 1$, we have 
\[
1 - e^{-t^\alpha} \leq 1,
\]
so
\[
\int_1^\infty \left(1 -e^{-t^\alpha}\right)t^{-\beta-1}\,dt \leq \int_1^\infty t^{-\beta-1}\,dt < \infty
\]
because $\beta > 0$. Therefore, $\mathbb{E}[S^\beta] < \infty$.
\end{proof}

\begin{proof}[Proof of Theorem \ref{thm: p=p0 reduction}] 
We will show the existence of a linear isometric embedding of $\ell^r(3)$ into an $L^1$ Bochner space with codomain space $\ell^p(3)$. 

Let
\[
\alpha = \frac{r}{p} \in (0, 1)
\]
and consider the probability space
\[
(\Omega, \, \mathcal{A}, \, \mathbb{P}) = ([0, \infty)^3, \, \mathcal{B}([0, \infty))^{\otimes3}, \, \nu_\alpha^{\otimes3})
\]
where $(S_1, S_2, S_3)$ is a random variable with this distribution. Stated more explicitly, $S_1, \, S_2, \, S_3$ are independent $\alpha$-stable random variables with the same distribution as $S$. For arbitrary $a_1, a_2, a_3 \geq 0$, we have
\begin{align}
a_1S_1 + a_2S_2 + a_3S_3 \overset{d}{=} (a_1^\alpha + a_2^\alpha + a_3^\alpha)^{1/\alpha}S. \label{eq: stable sum}
\end{align}
This can be seen by verifying the Laplace transforms of both sides and checking that they are equal:
\begin{align*}
\mathbb{E}[e^{-t(a_1S_1 + a_2S_2 + a_3S_3)}] = \prod_{j=1}^3 \mathbb{E}[e^{-ta_jS_j}] = \prod_{j=1}^3 e^{-(ta_j)^{\alpha}} = \exp\left(-t^\alpha\sum_{j=1}^3a_j^\alpha\right) = \mathbb{E}[e^{-t(a_1^\alpha + a_2^\alpha + a_3^\alpha)^{1/\alpha}S}],
\end{align*}
so both sides of \eqref{eq: stable sum} have the same distribution.

Define
\[
m = \mathbb{E}[S_j^{1/p}].
\]
Since $r > 1$, we have $1/p < \alpha$, so by Lemma \ref{lemma: finite mean} we have $m < \infty$. Moreover, since the specified law is not concentrated at zero, we have $m > 0$.

Now we define the operator 
\[
T: \ell^r(3) \to  L^1_{\ell^p(3)}(\Omega, \, \nu_\alpha^{\otimes3}),
\]
where each $x = (x_1, x_2, x_3) \in \ell^r(3)$ corresponds to the mapping $Tx: \Omega \to \ell^p(3)$ defined as
\[
(T x)(\omega) = \frac1m\left(S_1(\omega )^{1/p}x_1, S_2(\omega )^{1/p}x_2, S_3(\omega )^{1/p}x_3\right).
\]
It is easy to verify that $T$ is indeed linear. Moreover, for each fixed $x$, the $j$-th component is a measurable function in $\omega$, so $Tx$ is  measurable.

We claim that $T$ is an isometry. We have
\[
\lVert (T x)(\omega)\rVert_p = \frac1m\left(S_1(\omega) |x_1|^p + S_2(\omega)|x_2|^p + S_3(\omega)|x_3|^p\right)^{1/p},
\]
and applying \eqref{eq: stable sum} with $a_j = |x_j|^p$, we have
\begin{align*}
\sum_{j=1}^3 S_j |x_j|^p &\overset{d}{=} \left(\sum_{j=1}^3 \left(|x_j|^p\right)^\alpha\right)^{1/\alpha}S \\
&= \left(\sum_{j=1}^3 |x_j|^r\right)^{p/r}S \\
&= \lVert x \rVert_r^p S.
\end{align*}
So, by taking $p$-th roots, we get
\[
\left(\sum_{j=1}^3 S_j |x_j|^p\right)^{1/p} \overset{d}{=} \lVert x \rVert_rS^{1/p},
\]
and by taking expectations, we get
\[
\mathbb{E}\left[\left(\sum_{j=1}^3 S_j|x_j|^p\right)^{1/p}\right] = \lVert x \rVert_r \mathbb{E}[S^{1/p}] = m\lVert x \rVert_r.
\]
Thus,
\begin{align*}
\lVert Tx\rVert_{L^1(\Omega, \, \ell^p(3))} &= \int_\Omega \lVert (Tx)(\omega)\lVert_p \, d\,\mathbb{P}(\omega) \\
&= \int_\Omega \frac1m \left(\sum_{j=1}^3 S_j(\omega)|x_j|^p\right)^{1/p} \, d\,\mathbb{P}(\omega) \\
&= \frac1m \mathbb{E}\left[\left(\sum_{j=1}^3 S_j |x_j|^p\right)^{1/p}\right] = \lVert x \rVert_r.
\end{align*}
Hence, $T$ is a linear isometry, and there is a linear isometric embedding
\[
\ell^r(3) \hookrightarrow L^1_{\ell^p(3)}(\Omega, \, \nu_\alpha^{\otimes3}).
\]
Now if $\ell^p(3)$ is Hornich-Hlawka, then by Proposition \ref{prop: bochner hh}, the space $L^1_{\ell^p(3)}(\Omega, \, \nu_\alpha^{\otimes3})$ is also Hornich-Hlawka, so it follows that $\ell^r(3)$ too is Hornich-Hlawka.
\end{proof}

We conclude this subsection by proving the main reduction for Problem 1.

\begin{proof}[Proof of Theorem \ref{thm: endpoint}] The ``only if" direction is immediate, so we look at the ``if" direction. Assume $\ell^{p_0}(3)$ is Hornich-Hlawka. Corollary \ref{cor: p0} implies that $\ell^p(3)$ is Hornich-Hlawka for all $2 < p \leq p_0$. Corollary \ref{cor: n >= 3 hh} in turn implies that $\ell^p(N)$ is Hornich-Hlawka for all $2 < p \leq p_0$ and $N \geq 3$.
\end{proof}

\subsection{An endpoint conjecture} To reiterate the previous discussion, the only remaining regime to explore for the Hornich-Hlawka property of $\ell^p(N)$ is when ${N \geq 3}$ and $2 < p \leq p_0$. Corollary \ref{cor: n >= 3 hh} and Conjecture \ref{conj: lp0(3) not hh} would imply that $\ell^p(N)$ spaces in this regime are not Hornich-Hlawka. As evidence for Conjecture \ref{conj: lp0(3) not hh}, the authors of \cite{marinescu2025survey} cite a result of Dor \cite[Corollary to Theorem~1.5]{dor1976potentials}, which shows that $\ell^p(3)$ is not isometric to any subspace of $L^1(0, 1)$ for $p > 2$ (although, as discussed in Remark \ref{rem: false converse}, this does not preclude the possibility of $\ell^p(3)$ nevertheless being Hornich-Hlawka).

We propose the following endpoint conjecture, motivated primarily by informal exploratory numerical searches that did not yield a counterexample.

\begin{conjecture} \label{conj: lp0(3) hh}
The space $\ell^{p_0}(3)$ is Hornich-Hlawka.
\end{conjecture}

Theorem \ref{thm: endpoint} and Conjecture \ref{conj: lp0(3) hh} would imply that $\ell^p(N)$ is Hornich-Hlawka for all $N \geq 3$ and $2 < p \leq p_0$ and answer Problem 1 in the negative.

Independently of this work, there is an ongoing open collaborative project \cite{prove2me2026hlawka} for AI-assisted Lean formalizations. The project also conjectures that $\ell^p(N)$ is Hornich-Hlawka for $2 < p \leq p_0$, based on numerical evidence and the formalized statement of Theorem \ref{thm: all lpn} based on \cite{palomar-2026-09-25-000006-v1}. 

\subsection{Two convex-geometric reformulations of the Hornich-Hlawka inequality}

One may ask whether the Hornich-Hlawka property of a Banach space $X$ has a ``dual" form in relation to the dual space $X^*$. As it turns out, there are two such formulations admitting convex-geometric interpretations, one in terms of the intersection of four closed balls and another in terms of the Minkowski sum of \emph{symmetric lenses} as described below. Similar considerations had been conducted by Witsenhausen \cite{witsenhausen1973zonoid}, who gave a sufficient (though not necessary) condition on the dual space $X^*$.

\begin{theorem} \label{thm: four ball equivalence}
Let $X$ be a Banach space, and let $X^*$ be its dual. The following statements are equivalent:

\begin{enumerate}[(i)]
    \item The space $X$ is Hornich-Hlawka;
    \item For all $\phi, \, \psi, \, \chi \in X^*$ with $\lVert \phi \rVert_{X^*}, \, \lVert \psi \rVert_{X^*}, \, \lVert \chi \rVert_{X^*} \leq 1$, denote the four centers
    \[
    c_1 = -\phi - \psi - \chi, \quad c_2 = \phi + \psi - \chi, \quad c_3 = \phi - \psi + \chi, \quad c_4 = -\phi + \psi + \chi.
    \] 
    We have
    \[
    \bigcap_{i=1}^4 \overline{B_{X^*}(c_i, \, 2)} \neq \emptyset;
    \]
    \item For all $\phi, \, \psi, \, \chi \in X^*$ with $\lVert \phi \rVert_{X^*}, \, \lVert \psi \rVert_{X^*}, \, \lVert \chi \rVert_{X^*} \leq 1$ and for all $f_1, \, f_2,\,  f_3, \, f_4 \in X$ such that
    \[
    f_1 + f_2 + f_3 + f_4 = 0,
    \]
    we have
    \[
    \sum_{i=1}^4 c_i(f_i) \leq 2\sum_{i=1}^4 \lVert f_i \rVert_X.
    \]
    \item For all $f_1, \, f_2, \, f_3, \, f_4 \in X$ such that
    \[
    f_1 + f_2 + f_3 + f_4 = 0,
    \]
    we have
    \[
    \lVert -f_1 + f_2 + f_3 - f_4 \rVert_X + \lVert -f_1 + f_2 - f_3 + f_4 \rVert_X + \lVert -f_1 -f_2 + f_3 + f_4 \rVert_X \leq 2\sum_{i=1}^4 \lVert f_i \rVert_X.
    \]
    \item For $a \in X^*$, define the symmetric lens $L(a) = B_{X^*}(a, \, 1) \cap B_{X^*}(-a, \, 1)$. If $a, b \in X^*$ satisfy 
    \[
    \lVert a + b \rVert_{X^*} \leq 1, \quad \lVert a - b \rVert_{X^*} \leq 1,
    \]
    then we have 
    \[
    B_{X^*}(0, \, 1) \subseteq L(a) + L(b),
    \]
    where the right-hand side is understood to mean the Minkowski sum of $L(a)$ and $L(b)$.
\end{enumerate}
\end{theorem}

We will call statements (ii) the \emph{four-ball intersection property} and (v) the \emph{symmetric lens property}.

Related Rademacher/Hanner forms appear in earlier work on Hornich-Hlawka inequalities, such as in \cite{honda1998rademacher}. However, neither the particular four-ball intersection property nor the symmetric lens property on the dual space have been explicitly stated in the literature to the author's best knowledge.

We note that we do not identify $X$ with its dual space $X^*$ in Theorem \ref{thm: four ball equivalence}; we do not require reflexivity of $X$. We also stress that Theorem \ref{thm: four ball equivalence} does not imply that $X$ is Hornich-Hlawka if and only if $X^*$ is Hornich-Hlawka.

Our strategy is to prove the implications (i) $\Longrightarrow$ (ii) $\Longrightarrow$ (iii) $\Longrightarrow$ (iv) $\Longrightarrow$ (i) and (ii) $\Longleftrightarrow$ (v). Our proof of the first implication (i) $\Longrightarrow$ (ii) invokes Lemma \ref{lemma: bounded below operators}.

\begin{proof}[Proof of Theorem \ref{thm: four ball equivalence}] 
\textbf{[(i) $\Longrightarrow$ (ii)]} Assume that $X$ satisfies the Hornich-Hlawka inequality \eqref{eq: uvw rewrite}. Define a linear operator $T: \ell_3^1(X) \to \ell_4^1(X)$ by
\[
T(u, \, v, \, w) = (u + v + w, \, u + v - w, \, u - v + w, \, -u + v + w).
\]
We see that \eqref{eq: uvw rewrite} can be rewritten as
\[
\lVert T(u, \, v, \, w) \rVert_{\ell_4^1(X)} \geq 2 \lVert (u, \, v, \, w) \rVert_{\ell_3^1(X)}.
\]
We identify the dual spaces
\[
\left(\ell_3^1(X)\right)^* = \ell_3^\infty(X^*), \quad \left(\ell_4^1(X)\right)^* = \ell_4^\infty(X^*)
\]
and apply Lemma \ref{lemma: bounded below operators} to obtain that for all $(\phi, \, \psi, \, \chi) \in \ell_3^\infty(X^*)$, there is $(\alpha_1, \, \alpha_2, \, \alpha_3, \, \alpha_4) \in \ell_4^\infty(X^*)$ such that $T^*(\alpha_1, \, \alpha_2, \, \alpha_3, \, \alpha_4) = (\phi,\, \psi, \, \chi)$ and
\[
\max_{1 \leq i \leq 4} \;\lVert \alpha_i\rVert_{X^*} \leq \frac12 \max\{\lVert \phi \rVert_{X^*}, \, \lVert \psi \rVert_{X^*}, \, \lVert \chi \rVert_{X^*} \}.
\]
In particular, if $\lVert \phi \rVert_{X^*}, \lVert \psi \rVert_{X^*},\, \lVert \chi \rVert_{X^*} \leq 1$, then 
\[
\max_{1 \leq i \leq 4} \; \lVert \alpha_i \rVert_{X^*} \leq \frac12.
\]
Writing out $T^*$, we have
\begin{align*}
\alpha_1 + \alpha_2 + \alpha_3 - \alpha_4 &= \phi, \\
\alpha_1 + \alpha_2 - \alpha_3 + \alpha_4 &= \psi, \\
\alpha_1 - \alpha_2 + \alpha_3 + \alpha_4 &= \chi.
\end{align*}
Introducing the variable $\theta = \alpha_1 - \alpha_2 - \alpha_3 - \alpha_4$, we have
\begin{align*}
\alpha_1 &= \frac{\phi + \psi + \chi + \theta}4, \\
\alpha_2 &= \frac{\phi + \psi - \chi - \theta}4, \\
\alpha_3 &= \frac{\phi - \psi + \chi - \theta}4, \\
\alpha_4 &= \frac{-\phi + \psi + \chi - \theta}4.
\end{align*}
This means that
\begin{align*}
\lVert \theta - c_1 \rVert_{X^*} = \lVert \phi + \psi + \chi + \theta \rVert_{X^*} = 4 \lVert \alpha_1 \rVert_{X^*} \leq 2, \\
\lVert c_2 - \theta \rVert_{X^*} = \lVert \phi + \psi - \chi - \theta \rVert_{X^*} = 4 \lVert \alpha_2 \rVert_{X^*} \leq 2, \\
\lVert c_3 - \theta \rVert_{X^*} = \lVert \phi - \psi + \chi - \theta \rVert_{X^*} = 4 \lVert \alpha_3 \rVert_{X^*} \leq 2, \\
\lVert c_4 - \theta \rVert_{X^*} = \lVert -\phi + \psi + \chi - \theta \rVert_{X^*} = 4 \lVert \alpha_4 \rVert_{X^*} \leq 2,
\end{align*}
so
\[
\theta \in \bigcap_{i=1}^4 \overline{B_{X^*}(c_i, \, 2)}
\]
and hence (ii) holds.

\textbf{[(ii) $\Longrightarrow$ (iii)]}
Assume (ii) holds and assume $f_1, \, f_2, \, f_3, \,f_4 \in X$ satisfy $f_1 + f_2 + f_3 + f_4 = 0$. Let
\[
\theta \in \bigcap_{i=1}^4 \overline{B_{X^*}(c_i, 2)}.
\]
We have that $c_i - \theta \in \overline{B_{X^*}(0, 2)}$, and we also have
\[
\sum_{i=1}^4 \theta(f_i) = \theta(0) = 0.
\]
Hence, by the dual-norm inequality we get
\[
\sum_{i=1}^4 c_i(f_i) = \sum_{i=1}^4 \left[\theta(f_i) + (c_i - \theta)(f_i)\right] = \sum_{i=1}^4 (c_i - \theta)(f_i) \leq \sum_{i=1}^4 \lVert c_i - \theta \rVert_{X^*} \lVert f_i \rVert_X  \leq 2\sum_{i=1}^4 \lVert f_i \rVert_X,
\]
which is exactly (iii).

\textbf{[(iii) $\Longrightarrow$ (iv)]} Assume (iii) holds. Expanding out and rearranging the left-hand side of the inequality in (iii), we get
\begin{align*}
\sum_{i=1}^4 c_i(f_i) =  \phi( -f_1 + f_2 + f_3 - f_4) + \psi(-f_1 + f_2 - f_3 + f_4) +\chi( -f_1 -f_2 + f_3 + f_4).
\end{align*} 
Define
\begin{align*}
F &= -f_1 + f_2 + f_3 - f_4, \\
G &=  -f_1 + f_2 - f_3 + f_4, \\
H &= -f_1 -f_2 + f_3 + f_4.
\end{align*}
By (iii), we have
\[
\phi(F) + \psi(G) + \chi(H) \leq 2 \sum_{i=1}^4 \lVert f_i \rVert_X.
\]
Since $\phi, \psi, \chi$ range independently over the unit ball of $X^*$, we have by duality
\begin{align*}
\lVert F \rVert_X  + \lVert G \rVert_X + \lVert H \rVert_X  &= \sup_{\lVert \phi \rVert_{X^*} \leq 1} \phi(F) + \sup_{\lVert \psi \rVert_{X^*} \leq 1} \psi(G) + \sup_{\lVert \chi \rVert_{X^*} \leq 1} \chi(H) \\
&= \sup_{\lVert \phi \rVert_{X^*}, \, \lVert \psi \rVert_{X^*}, \, \lVert \chi \rVert_{X^*} \leq 1} \left[\phi(F) + \psi(G) + \chi(H)\right] \leq 2 \sum_{i=1}^4 \lVert f_i \rVert_X,
\end{align*}
which is precisely (iv).

\textbf{[(iv) $\Longrightarrow$ (i)]} Let $f_1,\, f_2,\, f_3, \,f_4 \in X$ be a zero-sum quadruple. Writing
\[
f_1 = x, \quad f_2 = y, \quad f_3 = z, \quad f_4 = -x-y-z,
\]
(iv) becomes
\[
2\lVert y + z \rVert_X + 2\lVert z + x \rVert_X + 2\lVert x + y \rVert_X \leq 2\left(\lVert x \rVert_X + \lVert y \rVert_X + \lVert z \rVert_X + \lVert x + y + z \rVert_X\right).
\]
Dividing both sides by 2 gives us the Hornich-Hlawka inequality \eqref{eq: hh inequality}.

\textbf{[(ii) $\Longrightarrow$ (v)]} Assume (ii) holds. Let $\chi \in B_{X^*}(0, 1)$ and let $\phi = a + b$ and $\psi = a - b$. By the hypothesis of (v), both $\phi$ and $\psi$ have norm at most 1 in $X^*$. Also, we have
\[
c_1 = -2a - \chi, \quad c_2 = 2a - \chi, \quad c_3 = 2b + \chi, \quad c_4 = -2b + \chi.
\]
Let $\theta \in X^*$ be a point in the intersection of the four balls as in the conclusion of (ii), and define
\[
\xi = \frac{\chi + \theta}2, \quad \eta = \frac{\chi - \theta}2.
\]
We have
\[
\theta - c_1 = 2(a + \xi), \quad \theta - c_2 = 2(\xi - a), \quad \theta - c_3 = -2(b + \eta), \quad \theta - c_4 = 2(b - \eta),
\]
all of which have norm at most 2 in $X^*$. Hence,
\[
\lVert a + \xi \rVert_{X^*} \leq 1, \quad \lVert a - \xi \rVert_{X^*} \leq 1, \quad \lVert b + \eta \rVert_{X^*} \leq 1, \quad  \lVert b - \eta \rVert_{X^*} \leq 1.
\]
Thus $\xi \in L(a)$ and $\eta \in L(b)$, and we conclude
\[
\chi = \xi + \eta \in L(a) + L(b).
\]

\textbf{[(v) $\Longrightarrow$ (ii)]} Assume that (v) holds, and let $\phi, \psi, \chi \in X^*$ have norms at most 1. Denote 
\[
a = \frac{\phi + \psi}2, \quad b = \frac{\phi - \psi}2.
\]
We see that $a + b = \phi$ and $a - b = \psi$, so these have norms at most 1. Thus we conclude by (v) that $\chi \in L(a) + L(b)$, so there are $\xi \in L(a)$ and $\eta \in L(b)$ such that
$\chi = \xi + \eta$.

We claim that $\theta = \xi - \eta$ lies in the intersection of the four balls as in the conclusion of (ii). Indeed, we compute as before
\[
\theta - c_1 = 2(a + \xi), \quad \theta - c_2 = 2(\xi - a), \quad \theta - c_3 = -2(b + \eta), \quad \theta - c_4 = 2(b - \eta),
\]
all of which have norm at most 2 by construction.
\end{proof}

Although this subsection proceeded with full generality of the Banach space $X$ without added geometrical assumptions, the motivation for Theorem \ref{thm: four ball equivalence} is the hope that one could resolve Conjecture \ref{conj: lp0(3) hh} by casting the conjecture in terms of its dual space $\ell^{q_0}(3)$, where $1/p_0 + 1/q_0 = 1$. Though this direction seems promising, we currently lack an argument that $\ell^{q_0}(3)$ satisfies the four-ball intersection property or the symmetric lens property.

\section{Problem 2: A Banach space without a three-dimensional Hornich-Hlawka subspace} 
Let us first restate Problem 2 as a conjecture.

\begin{conjecture}\cite[Problem~2]{marinescu2025survey} \label{prob: far apart}
There is a Banach space $X$ with the property that every finite-dimensional subspace of $X$ with dimension greater than or equal to $3$ is not Hornich-Hlawka.
\end{conjecture}
Taken literally, Conjecture \ref{prob: far apart} is trivially true by taking $X$ to be a three-dimensional space that is not Hornich-Hlawka (e.g. $X = \ell^\infty(3)$). The problem becomes more interesting when we impose the additional hypothesis that $X$ is infinite-dimensional, or finite-dimensional with dimension at least 4. Indeed, Theorem \ref{thm: l^2 equivalent} answers the infinite-dimensional variant in the affirmative. The high-level intuition is that we construct the renorming on $\ell^2$ specifically such that every three-dimensional subspace has an $\ell^\infty$-type violation of Hornich-Hlawka.

For the rest of this section, we define for a seminorm $P$ on a Banach space $X$ the quantity
\[
\Phi_P(u, v, w) = P(u + v + w) + P(u + v - w) + P(u - v + w) + P(-u + v + w) - 2Pu - 2Pv - 2Pw,
\]
for $u, v, w \in X$.

\begin{remark} \label{rem: hh seminorm}
Extending the Hornich-Hlawka inequality to seminorms, we observe that \eqref{eq: uvw rewrite} on $X$ for the seminorm $P$ is equivalent to $\Phi_P(u, v, w) \geq 0$.
\end{remark}

We will first give a roadmap of the proof of Theorem \ref{thm: l^2 equivalent}.

\begin{proof}[Proof roadmap of Theorem \ref{thm: l^2 equivalent}]
We will fix a sequence $\{g_n\}_{n=1}^\infty$ dense on the unit sphere of $\ell^2$ and consider the countable collection of all ordered triples
\[
G_{ijk} = (g_i, g_j, g_k)
\]
whose members are linearly independent (where each $g_i$ is treated as a functional in the dual space $x \mapsto \langle x, g_i \rangle$). For every such triple $G_{ijk}$ and $1 < \tau < 2$, define
\[
M_{G_{ijk}, \tau}(x) = \max\left\{|\langle x, g_i \rangle|, \left|\frac{\langle x, g_j \rangle}{\tau}\right|, \left|\frac{\langle x, g_k \rangle}{\tau^2}\right|\right\}.
\]
We observe that each $M_{G_{ijk}, \tau}$ is a seminorm with $
M_{G_{ijk}, \tau}(x) \leq \lVert x \rVert_2$ for all $x \in \ell^2$, because each $g_n$ has norm 1 and $\tau > 1$.

Now enumerate the ordered triplets $G_{ijk}$ by reindexing them as $G_1, G_2, G_3, \ldots$ (this is possible since the collection of all $G_{ijk}$'s is countable). Consider the triangular schedule 
\[
G_1;
\]
\[
G_1, G_2;
\]
\[
G_1, G_2, G_3;
\]
\[
\vdots
\]
where each $G_n$ occurs infinitely often. We will call each occurrence of $G_n$ a \emph{visit to} $G_n$.

Our goal is to construct one global sequence $M_1, M_2, M_3, \ldots$ of seminorms of the form $M_{G_n, \tau}$ inductively based on this schedule; we give the explicit procedure as follows. At each visit to $G_n$ following this schedule, assume that we have already listed $N$ seminorms in the global sequence. Choose $L = 30N + 1$ distinct parameters 
\[
\tau_1, \ldots, \tau_L \in (1, 2)
\]
chosen differently from all parameters previously used, and append
\[
M_{G_n, \tau_1}, \ldots, M_{G_n, \tau_L}
\]
to the global sequence. Continuing this process indefinitely defines the global sequence $\{M_m\}_{m=1}^\infty$, where each \[
M_m = M_{G_{\sigma(m), \tau_m}}
\] corresponds to a pair $(G_{\sigma(m)}, \tau_m)$ (where the $\tau_m$'s are reindexed).

Introduce the numbers
\[
\varepsilon_m = 2^{-2^m}, \quad m \in \N,
\]
and for $x \in \ell^2$, define
\begin{align} \label{eq: x norm}
\lVert x \rVert_X = \lVert x \rVert_2 + \sum_{m=1}^\infty \varepsilon_m M_m(x).
\end{align}
Our proof will conclude that $\lVert x \rVert_X$ is a norm on $
\ell^2$ that is topologically equivalent to the usual norm, where the Hornich-Hlawka inequality is violated on every three-dimensional subspace of $X$ by Remark \ref{rem: hh seminorm}.
\end{proof}

We will now fill in the necessary details of the above proof roadmap. We will need many intermediate results; we will prove each of them explicitly.

\begin{lemma} \label{lemma: zero first variation}
Assume $P$ is a seminorm that is Fr\'echet differentiable on $X$ at $r \neq 0$. For a fixed $a, b, c \in X$, denote
\[
u_\delta = r + \delta a, \quad v_\delta = r + \delta b, \quad w_\delta = r + \delta c.
\]
Then at $r$,  the Fr\'echet derivative $DP$ is constant along positive rays. Moreover, the function $\Phi_P$ has zero first variation, i.e. for $\delta$ sufficiently small,
\[
\Phi_P(u_\delta , v_\delta , w_\delta ) = o(\delta).
\]
\end{lemma}

\begin{proof}
First, we verify that
$ DP(\lambda r ) = DP(r) $
for all $\lambda > 0$. We have, for small $h$,
\[
P(r + h) = P(r) + DP(r)[h] + o(\lVert h \rVert),
\]
so by positive homogeneity of $P$,
\begin{align*}
P(\lambda r + h) &= \lambda P\left(r + \frac{h}{\lambda}\right) \\
&= \lambda P(r) + \lambda DP(r)\left[\frac{h}{\lambda}\right] + \lambda o\left(\frac{\lVert h \rVert}{\lambda}\right) \\
&= P(\lambda r) + DP(r)[h] + o(\lVert h \rVert)
\end{align*}
so we conclude that $DP(\lambda r) = DP(r)$ for $\lambda > 0$. In particular, we have that $L := DP(r) = DP(3r)$.

Now we compute
\begin{align*}
&u_\delta + v_\delta + w_\delta = 3r + \delta(a + b + c), \quad  &u_\delta + v_\delta - w_\delta = r + \delta(a + b - c), \\
&u_\delta - v_\delta + w_\delta = r + \delta(a - b + c), \quad &-u_\delta + v_\delta + w_\delta = r + \delta(-a + b + c).
\end{align*}
This means that the total first-order contribution of the perturbation appearing in
\[
P(u_\delta + v_\delta + w_\delta) + P(u_\delta + v_\delta - w_\delta) + P(u_\delta - v_\delta + w_\delta) + P(-u_\delta + v_\delta + w_\delta)
\]
is
\begin{align*}
L(a + b + c) + L(a + b - c) + L(a - b + c) + L(-a + b + c) = 2L(a + b + c)
\end{align*}
by linearity of $L$. Meanwhile, the total first-order contribution of the terms
\[
-2P(u_\delta) -2P(v_\delta) -2P(w_\delta) 
\]
is 
\[
-2L(a) - 2L(b) - 2L(c) = -2L(a + b + c),
\]
so the first-order contributions in $\Phi_P$ cancel out. The constant terms also cancel out:
\[
P(3r) + 3P(r) - 3 \cdot 2P(r) = 0.
\]
\end{proof}

\begin{lemma} \label{lemma: max cusp}
Let $E \subseteq \ell^2$ be a three-dimensional subspace and let $F: E \to \R^3$ be a linear isomorphism. Define
\[
M : E \to \R,\quad M(x) = \lVert F(x)\rVert_\infty.
\]
Let
\[
r = F^{-1}(1, 1, 1), \quad a = F^{-1}e_1, \quad b = F^{-1}e_2, \quad c = F^{-1}e_3
\]
and for $0 < \delta < 1$ denote
\[
u_\delta = r + \delta a, \quad v_\delta = r + \delta b, \quad w_\delta = r + \delta c.
\]
Then $M$ is a seminorm on $E$ with
\[
\Phi_M(u_\delta, v_\delta,  w_\delta) = -2\delta.
\]
\end{lemma}

\begin{proof}
It is easy to verify that $M$ is a seminorm. We have
\[
F(u_\delta) = (1 + \delta, 1, 1), \quad F(v_\delta) = (1, 1 + \delta, 1), \quad F(w_\delta) = (1, 1, 1 + \delta),
\]
so
\[
M(u_\delta) = M(v_\delta) = M(w_\delta) = 1 + \delta,
\]
and moreover
\[
M(u_\delta + v_\delta + w_\delta) = \lVert (3 + \delta, 3 + \delta, 3 + \delta)\rVert_\infty = 3 + \delta.
\]
On the other hand,
\[
M(u_\delta + v_\delta - w_\delta) = \lVert (1 + \delta, 1 + \delta, 1 - \delta )\rVert_\infty = 1 + \delta
\]
and likewise $M(u_\delta - v_\delta + w_\delta) = M(-u_\delta + v_\delta + w_\delta) = 1 + \delta$. Therefore,
\[
\Phi_M(u_\delta, v_\delta, w_\delta) = 3(1 + \delta) + (3 + \delta) - 6(1 + \delta) = -2\delta.
\]
\end{proof}

\begin{lemma} \label{lemma: restrictions dense}
Let $\{g_n\}_{n=1}^\infty$ be a sequence dense in the unit sphere of $H = \ell^2$, again identifying each $g_n$ with the functional $
x \; \mapsto \; \langle x, g_n \rangle  $
in the dual space. Then the restrictions $g_n |_{E} \in E^*$ span $E^*$ for every three-dimensional subspace $E \subseteq \ell^2$.
\end{lemma}

\begin{proof}
Suppose not; the common annihilator of every $g_n|_E$ contains a nonzero $x \in E$. Thus $\langle x, g_n \rangle = 0$ for every $n \in \N$. By continuity of the inner product and density of $\{g_n\}_{n=1}^\infty$, we have $\langle x, h \rangle = 0$ for every unit vector $h \in \ell^2$. But 
\[
\left\langle x, \frac{x}{\lVert x \rVert_2}\right\rangle  = \lVert x \rVert_2 \neq 0,
\]
so this is a contradiction.
\end{proof}

\begin{lemma}\label{lemma: 15 planes}
Let $E$ be a three-dimensional Banach space and let $\phi, \psi, \chi \in E^*$. For $x \in E$, let
\[
M(x) = \max\{|\phi(x)|, |\psi(x)|, |\chi(x)|\}.
\]
Then the set of points on which $M$ is not Fr\'echet differentiable is contained in a union of at most 15 proper planes through the origin.
\end{lemma}

\begin{proof}
For notational convenience, denote
\[
\mathcal{A} = \{\phi, -\phi, \psi, -\psi, \chi, -\chi\}
\]
and write
\[
M(x) = \max_{\lambda \in \mathcal{A}} \lambda(x).
\]
We remove repetitions from $ \mathcal{A}$. If exactly one member of $\mathcal{A}$ attains the maximum at $x \in E$, then $M$ agrees with that functional on a neighborhood of $x$, and is therefore differentiable there. Consequently, $M$ is non-differentiable only when two distinct members of $\mathcal{A}$ agree there.

Assume distinct $\lambda, \mu \in \mathcal{A}$ satisfy $\lambda(x) = \mu(x)$. This equality set is the proper plane
\[
\ker (\lambda - \mu),
\]
of which there are at most $\binom{6}2 = 15$, corresponding to possible pairs $(\lambda, \mu)$. Hence, all points at which $M$ is non-differentiable is contained in the union of these planes.
\end{proof}

\begin{proposition} \label{prop: arbitrary index}
Let $E \subseteq \ell^2$ be a three-dimensional subspace. Choose an ordered triple $G = (g_i, g_j, g_k)$ such that the restricted map
\[
G_E: E \to \R^3, \quad G_E(x) = \left(\langle x, g_i \rangle, \langle x, g_j \rangle, \langle x, g_k \rangle\right)
\]
is an isomorphism. In the global sequence of seminorms as in the proof roadmap of Theorem \ref{thm: l^2 equivalent}, there exist arbitrarily large indices $m \in \N$ with the following properties:
\begin{enumerate}
    \item $M_m = M_{G, \tau}$ for some $\tau \in (1, 2)$;
    \item If we define
    \[
    F_\tau(x) = \left(\langle x, g_i \rangle, \frac{\langle x, g_j \rangle}{\tau}, \frac{\langle x, g_k \rangle}{\tau^2}\right), \quad M_{G, \tau}(x) = \lVert F_\tau(x) \rVert_\infty,  \quad r_\tau = G_E^{-1}(1, \tau, \tau^2),
    \]
    then every restriction $M_k|_E$, $k < m$, is Fr\'echet differentiable at $r_\tau$.
    \item If
    \[
    a_\tau = F_\tau^{-1}e_1, \quad b_\tau = F_\tau^{-1}e_2, \quad c_\tau = F_\tau^{-1}e_3,
    \]
    then there is a constant $C_E < \infty$ independent of $m$ and $\tau$ such that
    \[
    \lVert a_\tau \rVert_2 + \lVert b_\tau \rVert_2 + \lVert c_\tau \rVert_2 \leq C_E.
    \]
\end{enumerate}
\end{proposition}

\begin{proof}
Consider a visit to $G$ as in the proof roadmap of Theorem \ref{thm: l^2 equivalent}, and assume $N$ seminorms have been listed in the global sequence prior to that visit. By Lemma \ref{lemma: 15 planes}, the nondifferentiability sets of the restrictions
\[
M_1|_E, \ldots, M_N|_E
\]
are together contained in a union of at most $15N$ proper planes of $E$. Under the coordinates supplied by $G_E$, a proper plane has the algebraic equation
\[
\alpha x_1 + \beta x_2 + \gamma x_3 = 0
\]
for some constants $\alpha, \beta, \gamma$ not all zero.

We also have $G_E(r_\tau) = (1, \tau, \tau^2)$. Hence, the plane contains $r_\tau$ only when
\[
\alpha + \beta \tau + \gamma \tau^2 = 0.
\]
The polynomial on the left-hand side is nonzero and has at most two real roots. Thus, each plane excludes at most two candidate parameters for $\tau$, so the union of at most $15N$ planes excludes at most $30N$ candidate parameters. Since the visit to $G$ supplies $L = 30N + 1$ additional choices for $(G, \tau)$, at least one such candidate has the property that every one of
\[
M_1|_E, \ldots, M_N|_E
\]
is differentiable at $r_\tau$. The existence of the candidate resolves (1).

Now assume a different parameter $\sigma \neq \tau$ from the same batch occurs before the selected $\tau$. We have
\[
F_\sigma(r_\tau) =
\left(
1,\frac{\tau}{\sigma},
\left(\frac{\tau}{\sigma}\right)^2
\right).
\]
If $\sigma > \tau$, then the first coordinate is uniquely largest out of all three coordinates. Likewise, if $\sigma < \tau$, then the third coordinate is uniquely largest. Hence, the seminorm 
\[
M_{G, \sigma}|_E(x) = \lVert F_\sigma(x) \rVert_\infty
\]
is differentiable at $x = r_\tau$. Thus, every seminorm in the global sequence with index below $m$ is differentiable at $r_\tau$, satisfying (2).

Next, we see that
\[
F_\tau^{-1}(x_1, x_2, x_3) = G_E^{-1}(x_1, \tau x_2, \tau^2 x_3).
\]
Since $G_E$ is an isomorphism and $1 < \tau < 2$, we get
\begin{align*}
\lVert a_\tau \rVert_2 + \lVert b_\tau \rVert_2 + \lVert c_\tau \rVert_2 &= \lVert G_E^{-1}e_1 \rVert_2 + \lVert \tau G_E^{-1}e_2 \rVert_2 + \lVert \tau^2 G_E^{-1}e_3 \rVert_2  \\
&\leq \lVert G_E^{-1}e_1 \rVert_2 + 2\lVert G_E^{-1}e_2 \rVert_2 + 4\lVert  G_E^{-1}e_3 \rVert_2,
\end{align*}
and the right-hand side is a finite constant independent of $m$ and $\tau$, as required in (3).

Finally, since visits to $G$ happen infinitely often, we may take $m$ to be arbitrarily large.
\end{proof}

\begin{lemma} \label{lemma: tail estimate}
For the constants $\varepsilon_m$ as defined in the proof roadmap of Theorem \ref{thm: l^2 equivalent}, we have
\[
\sum_{m=1}^\infty \varepsilon_m < \infty \quad \text{ and } \quad \lim_{m \to \infty} \frac{\sum_{k=m + 1}^\infty \varepsilon_k}{\varepsilon_m} = 0.
\]
\end{lemma}

\begin{proof}
Both claims can be verified by observing that $2^j \geq j + 1$ for all $j \in \N$. For the first claim, we have
\[
\sum_{m=1}^\infty \varepsilon_m \leq \sum_{m=1}^\infty 2^{-m-1} < \infty.
\]
For the second claim, we have
\[
\frac{\sum_{k=m + 1}^\infty \varepsilon_k}{\varepsilon_m} = \sum_{j=1}^\infty 2^{-2^m(2^j - 1)} \leq \sum_{j=1}^\infty 2^{-j 2^m} = \frac{2^{-2^m}}{1 - 2^{-2^m}} \xrightarrow{m \to \infty} 0.
\]
\end{proof}

\begin{lemma} \label{lemma: lipschitz}
Every $M_m$ as defined in the proof roadmap of Theorem \ref{thm: l^2 equivalent} is 1-Lipschitz with respect to $\lVert \cdot \rVert_2$, i.e. for all $x, y \in \ell^2$, we have
\[
|M_m(x) - M_m(y)| \leq \lVert x - y \rVert_2.
\]
Moreover, if we let $\delta \geq 0$ and $r, a, b, c \in \ell^2$ and define 
\[
u_\delta = r + \delta a, \quad v_\delta = r + \delta b, \quad w_\delta = r + \delta c,
\]
then
\[
|\Phi_{M_m}(u_\delta, v_\delta, w_\delta)| \leq 6\delta\left(\lVert a \rVert_2 + \lVert b \rVert_2 + \lVert c \rVert_2 \right).
\]
\end{lemma}

\begin{proof}
Since $M_m$ is a seminorm, we  verify that
\[
| M_m(x) - M_m(y)| \leq | M_m(x - y)| \leq \lVert x - y \rVert_2,
\]
by the observation noted in the proof roadmap of Theorem \ref{thm: l^2 equivalent}. Implicitly, for every $h \in \ell^2$, we have
\[
|M_m(x + \delta h) - M_m(x)| \leq \delta \lVert h \rVert_2.
\]
Now write $S = \lVert a \rVert_2 + \lVert b \rVert_2 + \lVert c \rVert_2$ and denote
\[
A_\delta = M_m(u_\delta + v_\delta + w_\delta) + M_m(u_\delta + v_\delta - w_\delta) + M_m(u_\delta - v_\delta + w_\delta) + M_m(-u_\delta + v_\delta + w_\delta) 
\]
and
\[
B_\delta = 2M_m(u_\delta) + 2M_m(v_\delta) + 2M_m(w_\delta),
\]
so that $\Phi_{M_m}(u_\delta, v_\delta, w_\delta) = A_\delta - B_\delta$. At $\delta = 0$, positive homogeneity of $M_m$ gives
\[
A_0 = M_m(3r) + 3M_m(r) =6M_m(r), \quad B_0 = 3 \cdot 2M_m(r) = 6M_m(r),
\]
so we have $\Phi_{M_m}(u_\delta, v_\delta, w_\delta) = (A_\delta - A_0) - (B_\delta - B_0)$.

We next estimate each term of $A_\delta - A_0$. Estimation of the term
\[
M_m(u_\delta + v_\delta + w_\delta) - M_m(u_0 + v_0 + w_0) 
\]
yields
\[
|M_m(3r + \delta(a + b + c)) - M_m(3r)| \leq \delta \lVert a + b + c \rVert_2 \leq \delta S.
\]
The same upper bound holds for each of the other three analogous terms of $A_\delta - A_0$. Hence, $|A_\delta - A_0| \leq 4\delta S$.

We now estimate $B_\delta - B_0$. We have
\begin{align*}
|M_m(r + \delta a) - M_m(r)| &\leq \delta \lVert a \rVert_2, \\
|M_m(r + \delta b) - M_m(r)| &\leq \delta \lVert b \rVert_2, \\
|M_m(r + \delta c) - M_m(r)| &\leq \delta \lVert c \rVert_2
\end{align*}
and adding these three inequalities yields 
\[
|B_\delta - B_0| \leq 2\delta \lVert a \rVert_2 +  2\delta \lVert b \rVert_2 +  2\delta \lVert c \rVert_2 = 2\delta S.
\]
We conclude that 
\[
|\Phi_{M_m}(u_\delta, v_\delta, w_\delta)| \leq |A_\delta - A_0| + |B_\delta - B_0| \leq 6\delta S.
\]
\end{proof}

We are finally ready to lay out the proof of Theorem \ref{thm: l^2 equivalent}.

\begin{proof}[Proof of Theorem \ref{thm: l^2 equivalent}]

Our aim is to prove that the norm $\lVert \cdot \rVert_X$ as described in the proof roadmap satisfies the following: an arbitrary three-dimensional subspace $E$ of $\ell^2$ admits elements $u, v, w \in E$ such that $
\Phi_{\lVert \cdot \rVert_X}(u, v, w) < 0$. Remark \ref{rem: hh seminorm} will imply that $E$ is not Hornich-Hlawka.

Let $\{g_n\}_{n=1}^\infty$ be a dense sequence in the unit sphere of $\ell^2$. Following the proof roadmap of Theorem \ref{thm: l^2 equivalent}, perform the enumeration of the $G_{ijk}$'s and the construction of the global sequence $M_1, M_2, M_3, \ldots$.

Now let $E \subseteq \ell^2$ be a fixed three-dimensional subspace. 
By Lemma \ref{lemma: restrictions dense}, the restrictions $g_n|_E$ span $E^*$, so there exists a triplet $G_{ijk} =(g_i, g_j, g_k)$ such that
\[
G_E: E \to \R^3, \quad  G_E(x) = \left(\langle x, g_i \rangle, \langle x, g_j \rangle, \langle x, g_k \rangle\right)
\]
is an isomorphism. By Proposition \ref{prop: arbitrary index}, we may choose an index $m \in \mathbb{N}$ taken arbitrarily large so that statements (1) through (3) in that proposition are satisfied for this choice of $G_E$. 

For such an $m$ and associated parameter $\tau \in (1, 2)$, we take the definitions of $M_m(x)$ and $\lVert x \rVert_X$ as in  \eqref{eq: x norm}. By Lemma \ref{lemma: tail estimate} and the property that $ M_m(x) \leq \lVert x \rVert_2$, we have
\[
\lVert x \rVert_2 \leq \lVert x \rVert_X \leq \left(1 + \sum_{m=1}^\infty \varepsilon_m\right)  \lVert x \rVert_2,
\]
where the right-hand side is a finite constant multiple of $\lVert x \rVert_2$. This proves the assertion that $\lVert \cdot \rVert_X$ is equivalent to $\lVert \cdot \rVert_2$.

Now rewrite \eqref{eq: x norm} as
\[
\lVert x \rVert_X = \lVert x \rVert_2 + \sum_{k=1}^{m-1} \varepsilon_kM_k(x) + \varepsilon_m M_m(x) + \sum_{k=m+1}^\infty \varepsilon_k M_k(x),
\]
and denote parts of  the expression on the right-hand side as the seminorms
\[
P_{<m}(x) = \lVert x \rVert_2 + \sum_{k=1}^{m-1}\varepsilon_kM_k(x), \quad P_{>m}(x) =\sum_{k=m+1}^\infty \varepsilon_k M_k(x).
\]
We first bound the tail term $P_{>m}(x)$ as follows. Let $r_\tau, u_\delta, v_\delta, w_\delta$ be defined as in the notation of Lemma \ref{lemma: max cusp}, where we take $F = F_\tau$. By linearity in the seminorm argument, we have
\[
\Phi_{P_{>m}}(u_\delta, v_\delta, w_\delta) = \sum_{k=m+1}^\infty \varepsilon_k \Phi_{M_k}(u_\delta, v_\delta, w_\delta),
\]
where the series on the right-hand side converges absolutely by Lemma \ref{lemma: lipschitz} and Proposition \ref{prop: arbitrary index}:
\[
\sum_{k=m+1}^\infty \varepsilon_k \left|\Phi_{M_k}(u_\delta, v_\delta, w_\delta)\right| \leq \sum_{k=m+1}^\infty \varepsilon_k \cdot 6 \delta \left(\lVert a_\tau \rVert_2 + \lVert b_\tau \rVert_2 + \lVert c_\tau \rVert_2 \right) \leq 6C_E \delta \sum_{k=m+1}^\infty \varepsilon_k.
\]
Since $m$ can be arbitrarily large, by Lemma \ref{lemma: tail estimate} we may take $m$ sufficiently large so that
\[
\sum_{k=m + 1}^\infty \varepsilon_k \leq \frac{1}{12C_E} \varepsilon_m.
\]
This implies
\begin{align} \label{eq: tail estimate}
\left|\Phi_{P_{>m}}(u_\delta,v_\delta,w_\delta)\right| \leq \frac12 \varepsilon_m  \delta.
\end{align}
Next, we inspect $P_{<m}(x)$. By Proposition \ref{prop: arbitrary index}, every $M_k|_E$ where $k < m$ is differentiable at $r_\tau$, and the Hilbert norm is differentiable at every nonzero point. Hence, $P_{<m}|_E$ is also differentiable at $r_\tau$, so Lemma \ref{lemma: zero first variation} gives
\[
\Phi_{P_{<m}}(u_\delta, v_\delta, w_\delta) = o(\delta).
\]
Choose $0 < \delta < 1$ sufficiently small so that
\begin{align} \label{eq: earlier estimate}
\left|\Phi_{P_{<m}}(u_\delta,v_\delta,w_\delta)\right| \leq \frac12\varepsilon_m\delta.
\end{align}
Also, by Lemma \ref{lemma: max cusp}, we have exactly
\begin{align} \label{eq: cusp estimate}
\varepsilon_m \Phi_{M_m}(u_\delta, v_\delta, w_\delta) = -2 \varepsilon_m \delta. 
\end{align}

Combining \eqref{eq: tail estimate}, \eqref{eq: earlier estimate}, and \eqref{eq: cusp estimate}, we have by linearity
\begin{align*}
\Phi_{\lVert \cdot \rVert_X}(u_\delta, v_\delta, w_\delta) &= \Phi_{P_{<m}}(u_\delta, v_\delta, w_\delta) + \varepsilon_m \Phi_{M_m}(u_\delta, v_\delta, w_\delta) + \Phi_{P_{>m}}(u_\delta, v_\delta, w_\delta) \\
&\leq \frac12 \varepsilon_m \delta - 2\varepsilon_m \delta + \frac12 \varepsilon_m \delta = -\varepsilon_m \delta < 0.
\end{align*}
We have found $u, v, w \in E$ such that $\Phi_{\lVert \cdot \rVert_X}(u, v, w) < 0$, so Remark \ref{rem: hh seminorm} lets us conclude that $E$ is not Hornich-Hlawka.
\end{proof}

\begin{corollary} If $Y \subseteq X$ is a subspace with $\dim Y \geq 3$, then $Y$ is not Hornich-Hlawka. Consequently, the construction of $X$ in Theorem \ref{thm: l^2 equivalent} is an affirmative construction for Conjecture \ref{prob: far apart}.
\end{corollary}

\begin{proof}
Suppose to the contrary that $\dim Y \geq 3$ and $Y$ is Hornich-Hlawka. Then $Y$ contains a three-dimensional subspace that is Hornich-Hlawka, contradicting Theorem \ref{thm: l^2 equivalent}. 
\end{proof}

\begin{remark}
We note that the renorming in Theorem \ref{thm: l^2 equivalent} and its proof can be made arbitrarily close to the Hilbert norm. For $c > 0$, define
\[
N_c(x) = \lVert x \rVert_2 + c\sum_{m=1}^{\infty}\varepsilon_mM_m(x).
\]
Then
\[
\lVert x \rVert_2 \leq N_c(x) \leq \left(1 + c\sum_{m=1}^\infty \varepsilon_m\right)\lVert x \rVert_2.
\]
By the same arguments as in the proof of Theorem \ref{thm: l^2 equivalent}, one can show that the norm $N_c$ fails Hornich-Hlawka for every three-dimensional subspace.
\end{remark}

As an immediate consequence to Theorem \ref{thm: l^2 equivalent}, we obtain the following.

\begin{corollary} \label{cor: l^2 finite}
Let $\lVert \cdot \rVert_X$ be the norm function on $\ell^2$ as given as in Theorem \ref{thm: l^2 equivalent} and $N \geq 3$. Consider the $N$-dimensional restricted space $\ell^2(N)$ given by the first $N$ coordinate directions of $\ell^2$. No three-dimensional subspace of $(\ell^2(N), \lVert \cdot \rVert_X)$ is Hornich-Hlawka.
\end{corollary}

Hence, the answer to the finite-dimensional variant of Problem \ref{prob: far apart} is also true, and there is a constructive example for every finite dimension $N \geq 3$.

\section{\texorpdfstring{Problems 3 and 4: Generalizations of $L^p$ spaces to Nakano and Sobolev spaces}{Problems 3 and 4: Generalization of Lp spaces to Nakano and Sobolev spaces}} Problems 3 and 4 relate to the Hornich-Hlawka property of Nakano and Sobolev spaces that generalize $L^p$ spaces. We restate the problems here for ease of reference.

\begin{problem} \label{prob: nakano} \cite[Problem~3]{marinescu2025survey} Is every Nakano space $L^{p(\cdot)}(0, 1)$ with a variable exponent $p(\cdot):(0, 1) \to [1, 2]$ a Hornich-Hlawka space?
\end{problem}

\begin{problem} \label{prob: sobolev} \cite[Problem~4]{marinescu2025survey} Which Sobolev spaces $W^{k, \, p(\cdot)}(0, 1)$ with $k \in \N \cup \{0\}$ and a variable exponent $p(\cdot): (0, 1) \to [1, 2]$ verify the Hornich-Hlawka inequality?
\end{problem}

\subsection{Nakano spaces}

Theorem \ref{thm: nakano counterexample} answers Problem \ref{prob: nakano} in the negative. The proof, which is primarily computational, is given below.

\begin{proof}[Proof of Theorem \ref{thm: nakano counterexample}]
Let 
\[
A_1 = \left(0, \frac15\right), \quad A_2 = \left(\frac15, \frac25\right), \quad A_3 = \left(\frac45, 1\right)
\]
be disjoint subsets of $(0, 1)$. Define the functions $f_1, f_2, f_3 \in L^{p(\cdot)}(0, 1)$ by
\[
f_1 = 5 \cdot \mathds{1}_{A_1}, \quad f_2 = 5 \cdot \mathds{1}_{A_2}, \quad f_3 = \sqrt{5} \cdot \mathds{1}_{A_3}.
\]
We first derive a formula for the Nakano norm of elements in $L^{p(\cdot)}(0, 1)$ that are linear combinations of $f_1, f_2, f_3$. Let $\alpha, \beta, \gamma$ be scalars not all zero. If $f = \alpha f_1 + \beta f_2 + \gamma f_3$, then for $\lambda > 0$ we have
\begin{align*}
\int_0^1 \left|\frac{f(t)}{\lambda}\right|^{p(t)}\,dt &= \int_0^{1/5} \left|\frac{5\alpha}{\lambda}\right|\,dt + \int_{1/5}^{2/5} \left|\frac{5\beta}{\lambda}\right|\,dt + \int_{4/5}^1 \left|\frac{\sqrt5\gamma}{\lambda}\right|^2\,dt \\
&= \frac15\left|\frac{5\alpha}{\lambda}\right| + \frac15 \left|\frac{5\beta}{\lambda}\right| + \frac15\left|\frac{\sqrt5\gamma}{\lambda}\right|^2 = \frac{|\alpha| + |\beta|}{\lambda} + \frac{|\gamma|^2}{\lambda^2}.
\end{align*}
Thus, $\lVert f \rVert_{p(\cdot)}$ is the unique positive $\lambda$ satisfying
\begin{align*}
&\frac{|\alpha| + |\beta|}{\lambda} + \frac{|\gamma|^2}{\lambda^2} = 1 \\
&\Longleftrightarrow \lambda^2  - (|\alpha| + |\beta|)\lambda - |\gamma|^2 = 0 \\
&\Longleftrightarrow \lambda = \frac{|\alpha| + |\beta| + \sqrt{(|\alpha| + |\beta|)^2 + 4|\gamma|^2}}{2}.
\end{align*}
Now we choose the three elements
\[
x = -2f_1 - f_3, \quad y = -2f_2, \quad z = 2f_2 + f_3.
\]
We compute
\[
x + y = -2f_1 - 2f_2 - f_3, \quad y + z = f_3, \quad z + x = -2f_1 + 2f_2, \quad x + y + z = -2f_1
\]
so
\begin{align*}
&\lVert x \rVert_{p(\cdot)} = \lVert z \rVert_{p(\cdot)} = 1 + \sqrt2, \quad \lVert y \rVert_{p(\cdot)} =\lVert x + y + z \rVert_{p(\cdot)} = 2, \\
&\lVert x + y \rVert_{p(\cdot)} = 2 + \sqrt{5}, \quad \lVert y + z \rVert_{p(\cdot)} = 1, \quad \lVert z + x \rVert_{p(\cdot)} = 4.
\end{align*}
Thus, the left-hand side of the Hornich-Hlawka inequality \eqref{eq: hh inequality} is
\[
(1 + \sqrt2) + 2 +(1 + \sqrt2) + 2 = 6 + 2\sqrt2 < 9
\]
and the right-hand side is
\[
(2 + \sqrt5) + 1 + 4 = 7 + \sqrt5 > 9.
\]
Since the left-hand side is less than the right-hand side, \eqref{eq: hh inequality} is violated and we conclude that $L^{p(\cdot)}(0, 1)$ is not Hornich-Hlawka.
\end{proof}

\begin{remark} \label{rem: different p()}
Although we defined $p(\cdot)$ in Theorem \ref{thm: nakano counterexample} to be continuous on $(0, 1)$, we only used the hypothesis that $p(\cdot)$ has constant value $1$ on $(0, 2/5)$ and constant value $2$ on $(4/5, 1)$; the intermediate domain $[2/5, 4/5]$ has no bearing on the verification of this construction. By modifying the behavior of $p(\cdot)$ on this intermediate domain, we may produce alternative constructions of $p(\cdot)$, as illustrated in the following two examples. This lets us rule out many more classes of possibilities for Nakano spaces of the form $L^{p(\cdot)}(0, 1)$ being Hornich-Hlawka.
\end{remark}

\begin{example} 
Define 
\[
\rho(s) = \begin{cases}
e^{-1/s},  &s > 0 \\
0, &s \leq 0
\end{cases}, \quad  \eta(s) = \frac{\rho(s)}{\rho(s) + \rho(1 - s)},
\]
and $p(t) = 1 + \eta(5t - 3)$. Then the space $L^{p(\cdot)}(0, 1)$ also fails  to be Hornich-Hlawka by exactly the same proof of Theorem \ref{thm: nakano counterexample}. Consequently, there is a $C^\infty(0, 1)$ function $p(\cdot)$ that fails to permit $L^{p(\cdot)}(0, 1)$ to be Hornich-Hlawka.
\end{example}

\begin{example}
Let $p(\cdot)$ be piecewise constant, attaining values $1$ and $2$ on the intervals $(0, \frac{7}{10})$ and $[\frac{7}{10}, 1)$ respectively. The proof of Theorem \ref{thm: nakano counterexample} again implies $L^{p(\cdot)}(0, 1)$ is not Hornich-Hlawka. Thus, even a piecewise constant function $p(\cdot):(0, 1) \to [1, 2]$ may fail to permit $L^{p(\cdot)}(0, 1)$ to be Hornich-Hlawka.
\end{example}

The above examples serve as motivation for the following conjecture on possible Nakano spaces of the form $L^{p(\cdot)}(0, 1)$ that are Hornich-Hlawka.

\begin{conjecture} \label{conj: nakano}
If the Nakano space $L^{p(\cdot)}(0, 1)$ with a variable exponent $p(\cdot): (0, 1) \to [1, 2]$ is Hornich-Hlawka, then $p(\cdot)$ is constant almost everywhere.
\end{conjecture}

Combined with Corollary \ref{cor: p <= 2}, Conjecture \ref{conj: nakano} would imply a full classification of Nakano spaces $L^{p(\cdot)}(0, 1)$ that are Hornich-Hlawka: explicitly, for measurable $p(\cdot):(0, 1) \to [1, 2]$, the space $L^{p(\cdot)}(0, 1)$ would be Hornich-Hlawka if and only if $p(\cdot)$ is essentially constant.

\subsection{Sobolev spaces}

Theorem \ref{thm: sobolev nakano} reduces Problem \ref{prob: sobolev} to a problem purely in Nakano spaces. This shows that the choice of $k$ has no effect on the Hornich-Hlawka property of $W^{k, \, p(\cdot)}(0, 1)$  provided that the norm is chosen as in \eqref{eq: sobolev norm}.

We will need an auxiliary result, whose notation we will reuse in the proof of Theorem \ref{thm: sobolev nakano}.

\begin{lemma} \label{lemma: integration estimate}
For $n \in \N$ and $t \in (0, 1)$, define the sign function
\[
\varepsilon_n(t) = (-1)^j, \quad \frac{j}{2n} \leq t < \frac{j + 1}{2n}, \quad j \in \{0, 1, \ldots, 2n - 1\}.
\]
For $h \in L^{p(\cdot)}(0, 1)$, define $I$ by 
\[
(Ih)(x) = \int_0^x h(t)dt
\]
and let $I^m$ denote its $m$-fold iteration. Then  for every integer $m \geq 1$,
\[
\lim_{n \to \infty} \lVert I^m(\varepsilon_n h) \rVert_{p(\cdot)} = 0.
\]
\end{lemma}

\begin{proof}
We first establish that $L^{p(\cdot)}(0, 1) \subseteq L^1(0, 1)$. Let $h \in L^{p(\cdot)}(0, 1)$ and choose $\lambda > \lVert h \rVert_{p(\cdot)}$. Then
\[
\int_0^1 \left|\frac{h(t)}{\lambda}\right|^{p(t)} \, dt \leq 1.
\]
Write $g = |h|/\lambda$. On the set $\{g > 1\}$, we have $g \leq g^{p(t)} $ because $p(t) \geq 1$. Hence,
\[
\int_0^1 g = \int_{\{g \leq 1\}} g +  \int_{\{g > 1\}} g \leq 1 +  \int_{\{g > 1\}} g(t)^{p(t)}\,dt \leq 2,
\]
so $\lVert h \rVert_1 \leq 2\lambda$. Taking $\lambda \searrow \lVert h \rVert_{p(\cdot)} $ proves the desired inclusion.

Next, we prove the claim for the special case $m = 1$. Our aim is to prove that
\[
\lim_{n \to \infty} \sup_{0 \leq x \leq 1} \left|\int_0^x \varepsilon_n(t)h(t)\,dt\right| = 0.
\]
Assume for the moment that $h$ is continuous on $[0, 1]$. Since $\varepsilon_n$ alternates signs on intervals of length $1/(2n)$, we have
\[
\int_{j/n}^{(j + 1)/n} \varepsilon_n(t)h(t)\,dt = \int_{j/n}^{j/n + 1/(2n)} \left[h(t) - h\left(t + \frac{1}{2n}\right)\right]\,dt.
\]
Let $\omega_h$ be the modulus of continuity of $h$. The absolute value of the above expression is bounded above by
\[
\frac1{2n}\omega_h\left(\frac1{2n}\right).
\]
Summing this over all intervals $[j/n, (j + 1)/n]$ fully contained in $[0, x]$ and adding the possible incomplete final interval, we get
\[
\sup_{0 \leq x \leq 1} \left|\int_0^x \varepsilon_n(t)h(t)\,dt\right| \leq \frac12 \omega_h\left(\frac1{2n}\right) + \frac{\lVert h \rVert_\infty}{n} \xrightarrow{n\to\infty}0.
\]
Now we drop the continuity assumption on $h \in L^1$. Choose $g$ to be continuous with $\lVert h - g \rVert_1 < \eta$. We have
\[
\sup_{0 \leq x \leq 1}  \left|\int_0^x \varepsilon_n(t)[h(t) - g(t)]\,dt\right| \leq \lVert h - g \rVert_1 < \eta.
\]
Since the primitive of $\varepsilon_n g$ tends uniformly to zero as $n \to \infty$, so does that of  $\varepsilon_n h$ by taking $\eta \searrow 0$.

Finally, we prove the claim for all $m \geq 1$. If $F_n = I(\varepsilon_n h)$, then $I^m(\varepsilon_ n h) = I^{m-1}F_n$. By the bound $| F_n(x) | \leq \lVert F_n \rVert_\infty$ and repeated integration, we get
\[
|I^{m-1}F_n(x) | \leq \frac{\lVert F_n \rVert_\infty}{(m - 1)!}x^{m-1},
\]
so
\[
\lVert I^{m-1}F_n \rVert_\infty \leq \frac{\lVert F_n \rVert_\infty}{(m - 1)!}.
\]
By our proof for $m = 1$, we see that the right-hand side tends to zero. Hence, $\lVert I^m(\varepsilon_nh) \rVert_\infty \to 0$.

We observe that for every nonzero bounded $g$, we have
\[
\lVert g \rVert_{p(\cdot)} \leq \lVert g \rVert_\infty,
\]
which can be seen by taking $\lambda  = \lVert g \rVert_\infty$ in the definition of the Luxemburg norm. This inequality also holds for $g$ zero. Now taking $g = I^m (\varepsilon_n  h)$ completes the proof.
\end{proof}

\begin{proof}[Proof of Theorem \ref{thm: sobolev nakano}]

The case $k = 0$ is immediate, so we assume $k \geq 1$.

First assume $X = L^{p(\cdot)}(0, 1)$ is Hornich-Hlawka. We observe that the finite $\ell^1$-sum $\ell^1_{k + 1}(X)$ is Hornich-Hlawka with norm
\[
\lVert (f_1, \ldots, f_{k + 1}) \rVert_{\ell^1_{k + 1}(X)} = \sum_{j=1}^{k + 1} \lVert f_j \rVert_{p(\cdot)}.
\]
Indeed, each of the seven terms in the Hornich-Hlawka inequality for $\ell^1_{k + 1}(X)$ can be broken up as a sum in terms of each coordinate. Each coordinate-wise Hornich-Hlawka inequality holds for $X$, so summing these inequalities gives exactly the Hornich-Hlawka inequality on $\ell^1_{k + 1}(X)$.

Now define the linear operator $J: W^{k, \, p(\cdot)}(0, 1) \to \ell_{k+1}^1(X)$ by
\[
Jf = \left(f, f', \ldots, f^{(k)}\right).
\]
We have
\[
\lVert Jf \rVert_{\ell^1_{k + 1}(X)} = \sum_{j=0}^k \lVert f^{(j)} \rVert_{p(\cdot)} = \lVert f \rVert_{k, \, p(\cdot)},
\]
so $J$ is an isometric embedding of $W^{k, p(\cdot)}(0, 1)$ into the Hornich-Hlawka space $\ell^1_{k + 1}(X)$. This shows that $W^{k, p(\cdot)}(0, 1)$ is also Hornich-Hlawka.

Now assume that there is a counterexample to the Hornich-Hlawka inequality on $L^{p(\cdot)}(0, 1)$; we will lift this to a counterexample on $ W^{k, \, p(\cdot)}(0, 1)$. We borrow the notation from Lemma \ref{lemma: integration estimate}. Let $f, g, h \in L^{p(\cdot)}(0, 1)$ be such a counterexample. For $n \in \N$ and $k \geq 1$, define
\[
F_n = I^k(\varepsilon_n f), \quad G_n = I^k(\varepsilon_n g), \quad H_n = I^k(\varepsilon_n h),
\]
where we note that $F_n, G_n, H_n$ all belong to $W^{k, p(\cdot)}(0, 1)$ because their highest-order derivatives are in $L^{p(\cdot)}(0, 1)$ and their lower-order derivatives are bounded primitives. We see that
\[
F_n^{(j)} = I^{k-j}(\varepsilon_n f), \quad j \in \{0, \ldots, k - 1\}
\]
and $F_n^{(k)} = \varepsilon_n f$, with analogous formulas for $G_n$ and $H_n$. Since $|\varepsilon_n| = 1$, we have
\[
\lVert F_n \rVert_{k, p(\cdot)} = \sum_{j=0}^{k - 1} \lVert I^{k-j}(\varepsilon_n f) \rVert_{p(\cdot)} + \lVert \varepsilon_n f \rVert_{p(\cdot)} \xrightarrow{n \to \infty} \lVert f \rVert_{p(\cdot)} 
\]
by Lemma \ref{lemma: integration estimate}. By the same argument, we have more generally that for any constants $\alpha, \beta, \gamma$,
\[
\lVert \alpha F_n + \beta G_n + \gamma H_n \rVert_{k, p(\cdot)}  \xrightarrow{n \to \infty} \lVert \alpha f + \beta g + \gamma h  \rVert_{p(\cdot)}.
\]
Applying this to each of the seven combinations occurring in the Hornich-Hlawka inequality \eqref{eq: hh inequality}, we arrive at a violation of the Hornich-Hlawka inequality with $F_n, G_n, H_n \in W^{k, p(\cdot)}(0, 1)$ for $n$ sufficiently large. 
\end{proof}

Theorem \ref{thm: sobolev nakano} implies that the answer to Problem \ref{prob: sobolev} is essentially the same as the classification problem of Nakano spaces, which Conjecture \ref{conj: nakano} would resolve if settled in the affirmative.

\section{\texorpdfstring{Problem 5: An $\alpha$-power generalization of the Hornich-Hlawka inequality}{Problem 5: An alpha-power generalization of the Hornich-Hlawka inequality}} The Hornich-Hlawka inequality \eqref{eq: hh inequality} for inner product spaces may also be extended in the following sense. \label{sec: problem 5}

\begin{problem}\cite[Problem~5]{marinescu2025survey} 
Let $(E, \langle \cdot, \cdot \rangle)$ be a nonzero inner product space with induced norm $\lVert x \rVert = \sqrt{\langle x, x \rangle}$. For which values of $\alpha \in (1, 2)$ does
\begin{align}
\lVert x \rVert^\alpha + \lVert y \rVert^\alpha + \lVert z \rVert^\alpha + \lVert x + y + z \rVert^\alpha \geq \lVert x + y \rVert^\alpha + \lVert  y + z \rVert^\alpha + \lVert z + x \rVert^\alpha \label{eq: alpha power}
\end{align}
hold for all $x, y, z \in E$? 
\end{problem}

Taking $\alpha = 1$ in \eqref{eq: alpha power} yields the usual Hornich-Hlawka inequality \eqref{eq: hh inequality} and taking $\alpha = 2$ yields the Fr\'echet equality \eqref{eq: frechet}, both of which hold for all $x, y, z \in E$. Indeed, the $\alpha = 2$ case and \cite[Theorem~9]{marinescu2025survey} establish that,  when $\alpha \in (0, 1] \cup \{2\}$, 
\eqref{eq: alpha power} holds for all $x, y, z \in E$.

The following two theorems show that in fact, no other positive value of $\alpha$ permits the statement to hold.

\begin{theorem} \label{thm: x=y=z} Let $X$ be a nonzero normed space. For all $\alpha \in (1, 2)$, \eqref{eq: alpha power} fails on $X$ by taking
$x = y = z \neq 0$.
\end{theorem}

\begin{proof}
Indeed, take $x = y = z$ to be a nonzero element of $X$. Substituting in \eqref{eq: alpha power}, we get
\begin{align*}
3\lVert x \rVert^\alpha + \lVert 3x \rVert^\alpha \geq 3\lVert 2x\rVert^\alpha \quad \Longleftrightarrow \quad (3 + 3^\alpha)\lVert x \rVert^\alpha \geq (3 \cdot 2^\alpha)\lVert x \rVert^\alpha \quad \Longleftrightarrow \quad 1 + 3^{\alpha-1} \geq 2^\alpha.
\end{align*}
Since $1 < \alpha < 2$, the map $t \mapsto t^{\alpha - 1}$ is strictly concave on $(0, \infty)$, so Jensen's Inequality gives
\[
1^{\alpha-1} + 3^{\alpha - 1} < 2 \cdot \left(\frac{1 + 3}2\right)^{\alpha - 1} = 2^\alpha.
\]
 Therefore, \eqref{eq: alpha power} fails.
\end{proof}

\begin{theorem} \label{thm: x=y=-z} Let $X$ be a nonzero normed space. For all $\alpha \in (2, \infty)$, \eqref{eq: alpha power} fails on $X$ by taking
$x = y = -z \neq 0$.
\end{theorem}

\begin{proof} Indeed, take $x = y\neq 0$ and $z = -x$ in \eqref{eq: alpha power} to get
\begin{align*}
3\lVert x \rVert^\alpha + \lVert -x \rVert^\alpha &\geq \lVert 2x \rVert^\alpha + 2\lVert 0 \rVert^\alpha \quad \Longleftrightarrow \quad 4\lVert x \rVert^\alpha \geq 2^\alpha \lVert x \rVert^\alpha \quad 
\Longleftrightarrow \quad 4 \geq 2^\alpha \quad \Longleftrightarrow \quad \alpha \leq 2.
\end{align*}
So if $\alpha > 2$, then \eqref{eq: alpha power} fails.
\end{proof}

We conclude that for $E$ a nonzero inner product space, the only positive values of $\alpha$ that obey \eqref{eq: alpha power} for every $x, y, z \in E$ are precisely $\alpha \in (0, 1] \cup \{2\}$.

\section*{Acknowledgments}
The author thanks Grigoriy Blekherman for fruitful discussions and constructive feedback. The author credits Jose Acevedo, Sebastian Debus, and Cordian Riener for introducing the author to the Hornich-Hlawka inequality, thereby prompting this line of research.

\section*{Declarations}

\subsection*{AI Disclosure} The author acknowledges the usage of AI technology for several aspects of this work: ChatGPT-5.6 Sol, ChatGPT-6 Astra, and Gemini 2.5 Pro were used for literature search, proof ideas, and a review of the manuscript.

An initial version of the proof of Theorem \ref{thm: all lpn} was generated by ChatGPT-5.6 Sol during a prompt to solve Theorem \ref{thm: n=3 reduction} directly; the author made substantial changes to its presentation. The same model generated the example in Theorem \ref{thm: l^2 equivalent}; the author kept the example but simplified its exposition, most notably by breaking up the proof into several lemmas and filling in details. Parts of the proofs of Theorem \ref{thm: p=p0 reduction}, Theorem \ref{thm: sobolev nakano}, and Lemma \ref{lemma: integration estimate} were discovered by ChatGPT-6 Astra.

\subsection*{Author Contributions} The entirety of the manuscript has been written and checked by the sole author, and all computations have been independently verified. The author takes full responsibility for all content of the paper.

\subsection*{Competing Interests} Nothing to declare.

\subsection*{Funding Information} Nothing to declare.

\subsection*{Data Availability Statement} Not applicable.

\printbibliography

\end{document}